\documentclass[a4paper]{article}
\usepackage{preamble}

\usepackage[style=alphabetic, 
maxbibnames=3, backend=biber, 
doi=false, isbn=false, url=false]{biblatex}
\title{Explicit Square Zero Obstruction Theory}
\author{Shaul Barkan}

\begin{document}

\maketitle
\tableofcontents
\section{Introduction}
Let $R \to S$ be a surjective homomorphism of commutative rings and let $M$ be an $S$-module.
Suppose we are interested in classifying $R$-modules $\tild{M}$ which lift $M$
in the sense that $S \otimes_R \widetilde{M} \simeq M$.
In general, even the existence question for such a lift can be quite challenging, owing to the fact that the multiplications on $S$ and on $J \coloneq \ker(R \to S)$ are often hard to disentangle.
Perhaps we are fortunate however, and the multiplication on $J$ is nilpotent so that we may break up $R \to S$ into a sequence $R=S_k \to \cdots \to S_2 \to S_1 \to S_0=S$ of
\textit{square zero extensions}, i.e.~such that $J_i \coloneqq \ker(S_{i} \to S_{i-1})$ satisfies $J_i \cdot J_i =0 \subseteq S_i$.
We may then approach the problem inductively, working on a single square zero extension at a time.
We thus assume without loss of generality that $R \to S$ 
is itself a square zero extension.
The special case where $M$ is flat goes back to Grothendieck \cite{grothendieck}, who proved the following:
\begin{enumerate}
    \item[($\star$)]
    There exists an obstruction 
    $\mfr{o}(M) \in \Ext^2_S(M,J\otimes_S M)$
    which vanishes if and only if there exists an $R$-module $\tild{M}$ lifting $M$ in the sense that
    $S \otimes_R \tild{M} \simeq M$.
    Equivalence classes of such lifts form a torsor under the abelian group $\Ext^1_S(M,J\otimes_S M)$.
\end{enumerate}

This result marked the dawn of algebraic deformation theory which has since been generalized and applied in a wide variety of different contexts. 
This is a vast subject with a long history, and hence we are unable to give a complete account. 
For related work see \cite{gerstenhaber1964,lichtenbaum1967cotangent,illusie2006complexe,donald1974,artamkin1989deformation,laudal1995non,ile2001obstructions,yau2005deformation,lurie2007derived,pstrkagowski2022abstract}.
In this paper we prove the analog of $(\star)$ in the general setting of stable (non-symmetric) monoidal \categories{}.
For the sake of concreteness we first explain our results in the setting of spectra $\Sp$.

Recall that \textit{\hl{split square zero extension}} of an $\bbE_1$-algebra spectrum $A \in \Alg:=\Alg_{\bbE_1}(\Sp)$ by an $(A,A)$-bimodule $M$, is an $A$-augmented $\bbE_1$-algebra $\hl{A \ltimes M} \in \Alg_{/A}$ whose underlying spectrum is $A \oplus M$ and whose multiplication is given upto homotopy by the following formula (see \cite[\S 7.4]{HA}):
\[ (a_1,m_1) \cdot (a_2,m_2) \simeq (a_1 \cdot a_2,a_1 \cdot m_2+m_1 \cdot a_2)\]
We write $\hl{\Der(A;M)} \coloneqq \Map_{\Alg_{/A}}(A,A \ltimes M)$ for the space of \hl{\textit{$\bbE_1$-derivations}} with coefficients in $M$.
For book keeping purposes we define a \textit{\hl{square zero datum}} to be a triple $\hl{(A,I,\eta)}$ consisting of an $\bbE_1$-algebra spectrum $A \in \Alg$, a bimodule $I \in \BMod{A}{A}$ and a $\Sigma I$-valued $\bbE_1$-derivation $\eta \colon A \to A \ltimes \Sigma I \in \Der(A;\Sigma I) $. 
A square zero datum $(A,I,\eta)$ has an \hl{\textit{associated square zero extension}} defined by the following pullback square in $\Alg_{/A}$
\[\begin{tikzcd}
	{\hl{A^\eta}} & A \\
	A & {A \ltimes \Sigma I}
	\arrow["\eta", from=2-1, to=2-2]
	\arrow[from=1-1, to=2-1]
	\arrow[from=1-1, to=1-2]
	\arrow["{\eta_0}", from=1-2, to=2-2]
	\arrow["\lrcorner"{anchor=center, pos=0.125}, draw=none, from=1-1, to=2-2]
\end{tikzcd}\]
where $\hl{\eta_0} \colon A \to A \ltimes \Sigma I$ denotes the \textit{\hl{trivial derivation}}, given by the inclusion of the first factor under the equivalence 
$A \ltimes \Sigma I \simeq A \oplus \Sigma I$.
Our goal in this paper is to describe  $\LMod_{A^\eta}$ in terms of the square zero datum $(A,I,\eta)$.
The application to obstruction theory will follow directly from this alternative description of $\LMod_{A^\eta}$.
We shall now explain our results in detail.

\begin{defn}
    Given $M \in \BMod{A}{A}$ we write $
    \hl{\theta^{M}} \colon A \to \Sigma M \in \BMod{A \ltimes M}{A \ltimes M}$ 
    for the edge map in the canonical cofiber sequence of $\left(A \ltimes M,A \ltimes M\right)$-bimodules    
    \[M\too A \ltimes M \too A \quad \in\BMod{A \ltimes M}{A \ltimes M}.\]
\end{defn}

\begin{war}\label{war:null}
    The map $\theta^M \colon A \to \Sigma M$ is quite deceptive.
    The $(A,A)$-bimodule splitting $A \ltimes M \simeq A \oplus M$
    means that once we forget the action and consider the resulting map in $\BMod{A}{A}$ it becomes canonically null homotopic.
    In particular $\theta^M$ induces the zero map on underlying spectra.
    Informally speaking, the information in
    $\theta^M$ 
    is entirely hidden in the discrepancy between the left and right actions.
\end{war}

\begin{notation}
    Recall that bimodules are functorial in pairs of algebras (see \cref{notation:bimodule-precise} for details), i.e. we have a functor:
    \[\BMod{(-)}{(-)} \colon \Alg^\op \times \Alg^\op \to \Cat_\infty, \qquad (R,S) \longmapsto \BMod{R}{S}.\]
    In particular for any pair of $\bbE_1$-algebra morphisms 
    $f \colon R_1 \to R_2$ and $g \colon S_1 \to S_2$
    we have a restriction of scalars functor:
    \[(f,g)^\ast \colon \BMod{R_2}{S_2} \too \BMod{R_1}{S_1}\]
    Informally, $(f,g)^\ast$ takes as input an $(R_2,S_2)$-bimodule $M$ and outputs the $(R_1,S_1)$-bimodule whose underlying object is $M$ but whose bimodule structure is obtained by restricting the action on the left along $f$ and on the right along $g$.
\end{notation}

\begin{defn}
    Let $(A,I,\eta)$, be a square zero datum.
    We define the \textit{\hl{(left) obstruction map}}
    of $(A,I,\eta)$ as
    \[\hl{\theta_\eta} \coloneqq  (\eta_0,\eta)^\ast(\theta^{\Sigma I}) \colon A \too \Sigma^2 I \quad \in
    \BMod{A}{A}.
    \]
    For a left $A$-module $X \in \LMod_A$ we define
    \[\hl{\theta_\eta(X)}\coloneqq \theta_\eta \otimes_A X \colon X \to \Sigma^2 I \otimes_A X \quad \in 
    \LMod_A.
    \]
    Since $\LMod_A$ is stable the mapping space $\Map_{\LMod_A}(X,\Sigma^2 I\otimes_A X)$ is pointed by the zero map. 
    We let $\hl{\Null(\theta_\eta(X))}$ denote the space paths
    $0 \simeq \theta_\eta(X) $ in
    $\Map_{\LMod_A}(X,\Sigma^2 I\otimes_A X)$.
    Note that $\Null(\theta_\eta(X))$ is naturally a torsor under the grouplike $\bbE_\infty$-space $\Omega_0  \Map_{\LMod_A}(X,\Sigma^2 I\otimes_A X) \simeq \Map_{\LMod_A}(X,\Sigma I\otimes_A X)$.
\end{defn}

We are ready to state the promised application to obstruction theory in the setting of connective ring spectra.

\begin{corA}[\cref{cor:obstruction-theory-intro} for $\calA^{\ge 0}=\calM^{\ge 0}=\Sp^{\ge 0}$]\label{corA:obstruction-intro-connective-spectra}
    Let $(A,I,\eta)$ be a square zero datum such that $A$ and $I$ are connective and let $X$ be a connective left $A$-module. 
    Then there is a canonical pullback square of \categories{}:
    \[\begin{tikzcd}
	{\Null(\theta_\eta(X))} & {\LMod_{A^\eta}} \\
	{\{X\}} & {\LMod_{A}.}
	\arrow[from=2-1, to=2-2]
	\arrow[from=1-1, to=2-1]
	\arrow[from=1-1, to=1-2]
	\arrow["{A \otimes_{A^\eta} (-)}", from=1-2, to=2-2]
	\arrow["\lrcorner"{anchor=center, pos=0.125}, draw=none, from=1-1, to=2-2]
    \end{tikzcd}\]
    In particular, we obtain the following consequences which mirror $(\star)$:
    \begin{enumerate}
    \item 
    The class
    $[\theta_\eta(X)] \in \pi_0 \Map_A(X,\Sigma^2 I\otimes_A X) \simeq \Ext^2_A(X,I\otimes_A X)$
    vanishes if and only if $X$ lifts to an $A^\eta$-module $\tild{X}\in \LMod_{A^\eta}$ in the sense that
    $A \otimes_{A^\eta} \tild{X} \simeq X$.
    \item 
    Equivalence classes of lifts as in $(1)$ form a torsor under
    $\pi_0\Map_A(X,\Sigma I\otimes_A X) \simeq \Ext^1_A(X,I\otimes_A X)$.
    \end{enumerate}
\end{corA}

Extending the theorem above to the case where $A$, $I$ and $X$ are not necessarily connective introduces
certain subtle features that do not
appear in the connective case.
Still, the connectivity assumptions in \cref{corA:obstruction-intro-connective-spectra}
can be removed at the cost of throwing away certain "bad" connected components of
$\Null(\theta_\eta(X))$.
We shall now describe the non-connective case in detail.
Recall that an $(A,A)$-bimodule $V \in \BMod{A}{A}$ defines an endofunctor of $\LMod_A$ via:
\[ V \otimes_A(-) \colon \LMod_A \too \LMod_A, \qquad X \longmapsto V \otimes_A X. \]
Accordingly, a morphism of $(A,A)$-bimodules $\alpha \colon W \to V$ gives rise to a functor:
\[ \alpha \otimes_A (-) \colon \LMod_A \to \Ar(\LMod_A), \qquad X \longmapsto \left(\alpha \otimes_A X \colon  W \otimes_A X \to V \otimes_A X\right). \]

\begin{defn}\label{defn:null-category-sqz}
    Let $(A,I,\eta)$ be a square zero datum.
    We define the \category{}
    $\hl{\Null_{\theta_\eta}(\LMod_A)}$ as the pullback:
    \[\begin{tikzcd}
    	{\Null_{\theta_\eta}(\LMod_A)} && {\LMod_A^{\times 2}} \\
    	{\LMod_A} && {\Ar(\LMod_A).}
    	\arrow[""{name=0, anchor=center, inner sep=0}, "{\theta_\eta \otimes_A(-)}"', from=2-1, to=2-3]
    	\arrow[from=1-1, to=2-1]
    	\arrow[from=1-1, to=1-3]
    	\arrow["{(X,Y) \mapsto \left(0\colon X \to Y\right)}", from=1-3, to=2-3]
    	\arrow["\lrcorner"{anchor=center, pos=0.125}, draw=none, from=1-1, to=0]
    \end{tikzcd}\]
    Informally, this is the \category{} of pairs $(X,h)$ where $X \in \LMod_A$ and $h \in \Null(\theta_\eta(X))$.
\end{defn}

\begin{defn}
    By \cref{war:null} the map
    $\theta_\eta \otimes_A A \colon A \to \Sigma^2 I \in \LMod_A$ 
    has a canonical null homotopy which we denote by $\hl{\gamma_A} \colon 0 \simeq \theta_\eta \otimes_A A$.
    The pair $(A,\gamma_A)$ defines an object of $\Null_{\theta_\eta}(\LMod_A)$.
\end{defn}

\begin{defn}
    Let $\Und \colon \LMod_A \to \Sp$ denote the forgetful functor.
    The map 
    \[
    \Und(\theta_\eta(X)) \colon \Und(X) \to \Und(\Sigma^2 I \otimes_A X) \quad \in \,
    \Sp.
    \]
    is null homotopic.   
    Indeed, a canonical such null homotopy is provided by the following diagram:
    \[\begin{tikzcd}
	& {\Und(A) \otimes \Und(X)} &&& {\Und(\Sigma^2 I) \otimes \Und(X)} \\
	{\Und(X)} &&&&& {\Und(\Sigma^2 I \otimes_A X).}
	\arrow[""{name=0, anchor=center, inner sep=0}, "{\Und(\theta_\eta) \otimes \Und(X)}"', curve={height=12pt}, from=1-2, to=1-5]
	\arrow[from=1-5, to=2-6]
	\arrow["{\Und(\theta_\eta(X))}"', from=2-1, to=2-6]
	\arrow[from=2-1, to=1-2]
	\arrow[""{name=1, anchor=center, inner sep=0}, "0", curve={height=-18pt}, from=1-2, to=1-5]
	\arrow["{\Und(\gamma_A)\otimes \Und(X)}"{description, pos=0.45}, shorten <=4pt, shorten >=4pt, Rightarrow, from=1, to=0]
    \end{tikzcd}\]
    We write
    $\hl{\gamma_A(X)}: 0 \simeq \Und(\theta_\eta(X))$ for the corresponding path in $\Map_{\Sp}(\Und(X),\Und(\Sigma^2 I \otimes_A X))$.
\end{defn}

Recall that for a fixed $X \in \LMod_A$ the space
$\Null(\theta_\eta(X))$ 
is naturally a torsor under $\Map_{\LMod_A}(X,\Sigma I\otimes_A X)$.
In particular, given $h_1,h_2 \in \Null(\theta_\eta(X))$ 
we may define
$\hl{h_2-h_1}\coloneqq  h_1^{-1}\circ h_2 \in \Map_{\LMod_A}(X,\Sigma I\otimes_A X)$.
This notation is consistent with the torsor action in the sense that 
$h_1 + (h_2-h_1) \simeq h_2$.

\begin{defn}
    Given $(X,h) \in \Null_{\theta_\eta}(\LMod_A)$ 
    we define
    \[\hl{\beta_h} \coloneqq  \gamma_A(X) - \Und(h) \colon \Und(X) \too \Und(\Sigma I \otimes_A X) \quad \in \, \Sp. 
    \]
    \begin{enumerate}
    \item $(X,h)$ is called \hl{\textit{$\beta$-divisible}} if $\beta_h : \Und(X) \to \Und(\Sigma I\otimes_A X)$ is an equivalence.
    \item
    $(X,h)$ is called \hl{\textit{$\beta$-torsion}} if it admits no non-trivial maps into $\beta$-divisible objects, i.e. for any morphism $\psi \colon (X,h) \to (Y,g)$ in $\Null_{\theta_\eta}(\LMod_A)$ whose target $(Y,g)$ is $\beta$-divisible we have $\psi \simeq 0$.
    \end{enumerate}
\end{defn}

\begin{thmA}[\cref{thm:sqz-general-case} for $\calA=\calM=\Sp$]\label{thm:spectra-sqz-modules}
    Let $(A,I,\eta)$ be a square zero datum.
    There exists a canonically commuting square of \categories{}
    \[\begin{tikzcd}
	{\LMod_{A^\eta}} &&& {\LMod^{\times 2}_A} \\
	{\LMod_A} &&& {\Ar(\LMod_A),}
	\arrow["{\theta_\eta \otimes_A(-)}"', from=2-1, to=2-4]
	\arrow["{A\otimes_{A^\eta}(-)}"', from=1-1, to=2-1]
	\arrow["{(A\otimes_{A^\eta}(-),\Sigma^2 I \otimes_{A^\eta}(-))}", from=1-1, to=1-4]
	\arrow["{(X,Y) \mapsto \left(0\colon X \to Y\right)}", from=1-4, to=2-4]
    \end{tikzcd}\]
    which gives rise to a fully faithful embedding
    \[ 
    \LMod_{A^\eta} \hookrightarrow \Null_{\theta_\eta}(\LMod_A) =
    \LMod_A \times_{\Ar(\LMod_A)} \LMod^{\times 2}_A,
    \]
    whose essential image is precisely the $\beta$-torsion objects.
\end{thmA}

\cref{thm:spectra-sqz-modules}
is an instance of the more general \cref{thm:sqz-general-case}, whose formulation we now sketch.
Let $\PrLst$ denote the \category{} whose objects are stable presentable \categories{} and whose morphisms are left adjoint functors.
Fix a stable presentably monoidal \category{} $\calA \in \Alg(\xPrL{}{\st})$ and an $\bbE_1$-algebra 
$A \in \Alg(\calA)$.
By \cite[Theorem 7.3.4.13]{HA} we have a canonical equivalence 
$\BMod{A}{A}(\calA) \simeq \Sp(\Alg(\calA)_{/A})$.
Using this equivalence we can define the \textit{\hl{split square zero extension}} functor in this general setting as the composite:
\[\hl{A \ltimes(-)} \colon \BMod{A}{A}(\calA) \simeq \Sp(\Alg(\calA)_{/A}) \xrightarrow{\,\,\Omega^\infty \,\,} \Alg(\calA)_{/A}.\]
Just as before we define a square zero datum in $\calA$ to be a triple $(A,I,\eta)$ consisting of an $\bbE_1$-algebra $A \in \Alg(\calA)$, a bimodule $I \in \BMod{A}{A}(\calA)$ and an $\bbE_1$-derivation $\eta \in \Der(A;\Sigma I)\coloneqq  \Map_{\Alg(\calA)_{/A}}(A,A \ltimes \Sigma I)$.
Similarly, we define $A^\eta$ to be the pullback of 
$\left(A \xrightarrow{\eta} A \ltimes \Sigma I \xleftarrow{\eta_0}A\right)$ taken in $\Alg(\calA)_{/A}$.
Finally, recall that given a presentable left $\calA$-module $\mcal{M} \in \LMod_\calA(\xPrL{}{\st})$ and an $\bbE_1$-algebra $R \in \Alg(\calA)$ we may consider the \category{} $\LMod_R(\mcal{M})$ of left $R$-modules in $\mcal{M}$ \cite[Definition 4.2.1.13.]{HA}.

\cref{thm:sqz-general-case}
is a generalization of \cref{thm:spectra-sqz-modules} in which $\Sp$ is replaced by $\calA$ and $\LMod_{A^\eta}$ is replaced by $\LMod_{A^\eta}(\calM)$.
Such a formulation is essentially what we prove in
\cref{thm:sqz-general-case}.
To recover \cref{thm:spectra-sqz-modules} one simply substitutes $\calA=\calM=\Sp$.
We also prove a variant of \cref{corA:obstruction-intro-connective-spectra} in a similarly general setting.
For this we need an appropriate notion of connectivity in the stable \categories{} $\calA$ and $\calM$.
Such a notion is provided by pair of compatible $t$-structures on $\calA$ and $\calM$.
The data of a stable \category{} with a (right complete) $t$-structure is equivalent to a prestable \category{}.
For convenience sake, we formulate the result in terms of prestable \categories{} rather than $t$-structures.
We briefly recall the relevant terminology.

Recall that an \category{} is called \hl{\textit{prestable}} if it arises as the connective part of a $t$-structure on some stable \category{}.
To solidify this intuition we denote prestable \categories{} with a "${\ge\! 0}$" subscript, e.g. $\calC^{\ge 0}$.
Given a prestable \category{} $\calC^{\ge 0}$ the functor $\Sigma^\infty \colon \calC^{\ge 0} \too \Sp(\calC^{\ge 0})$ is fully faithful and its essential image is closed under extension.\footnote{in fact, this property characterizes prestable \categories{}.}
We then write $\calC \coloneqq  \Sp(\calC^{\ge 0})$ for the stabilization and $\calC^\heart \subseteq \calC^{\ge 0}$ for the subcategory of discrete, i.e. $0$-truncated, objects.
This notation makes manifest that $\calC^{\ge 0}$ is the connective part of a (unique) $t$-structure on $\calC$ with heart $\calC^\heart$.
A prestable \category{} $\calC^{\ge 0}$ is called \textit{\hl{separated}} if it has no non-zero $\infty$-connected objects.
Equivalently,$\calC^{\ge 0}$ is separated if and only if the homotopy groups functor $\pi_\ast^\heart \colon \calC \to \gr(\calC^\heart)$ is conservative. 
We are finally ready to state the connective variant of \cref{thm:sqz-general-case}.

\begin{thmA}[\cref{thm:connective-main-theorem}]\label{thm:sqz-connective-variant}
    Let $\calA^{\ge 0}$ be a presentably monoidal prestable \category{}, let $(A,I,\eta)$ be a square zero datum in $\calA^{\ge 0}$ and let $\mcal{M}^{\ge 0}$ be a seprated, prestable presentable left $\calA^{\ge 0}$-module.
    Then there is a canonical equivalence:
    \[ \LMod_{A^\eta}(\mcal{M}^{\ge 0})\iso \Null_{\theta_\eta}\left(\LMod_A(\mcal{M}^{\ge 0})\right)\]
\end{thmA}

\cref{thm:sqz-connective-variant} gives rise to the promised obstruction theory for modules along square zero extensions.

\begin{corA}\label{cor:obstruction-theory-intro}
    In the context of \cref{thm:sqz-connective-variant}, any bounded below left $A$-module $X \in \LMod_A(\mcal{M}^{>-\infty})$ participates in a pullback square:
	\[\begin{tikzcd}
	{\fib_0\left(\pt \xrightarrow{\theta_\eta(X)} \Map_{\LMod_A(\mcal{M})}(X,\Sigma^2 I \otimes_A X) \right)} && {\LMod_{A^\eta}(\calM).} \\
	\\
	{\{X\}} && {\LMod_A(\calM).}
	\arrow[from=1-3, to=3-3]
	\arrow[""{name=0, anchor=center, inner sep=0}, from=3-1, to=3-3]
	\arrow[from=1-1, to=3-1]
	\arrow[from=1-1, to=1-3]
	\arrow["\lrcorner"{anchor=center, pos=0.125}, draw=none, from=1-1, to=0]
    \end{tikzcd}\]
\end{corA}

Note that \cref{corA:obstruction-intro-connective-spectra} can be recovered from \cref{cor:obstruction-theory-intro} by setting $\calA = \calM =\Sp$. 

\subsubsection{Outline of the paper}


In \textit{\cref{section:2}}, we focus entirely on split square zero extensions.
After establishing some formal prerequisites, we study the relation between split square zero extensions and colax fixed points.
The technical backbone of this section is \cref{thm:fixcolax-categorifies-exterior} which, figuratively speaking, says that colax fixed points categorify split square zero extensions.
We then deduce \cref{cor:bifiber-seq-general-case} which is essentially a reformulation of \cref{thm:sqz-general-case} in the split case.

In \textit{\cref{section:3}}, we treat general case of non-split extensions.
We begin with a systematic analysis of \categories{} of the form 
$\Null_\alpha(\calD)$ 
where $\calD$ is a presentable stable \category{} and $\alpha\colon \Id \to E$ is some natural transformation in $\Fun_{\PrLst}(\calD,\calD)$.
We then specialize to the case of interest, namely $\calD \coloneqq \LMod_A(\calM), E\coloneq \Sigma^2 I\otimes_A(-)$ and $\alpha \coloneqq \theta_\eta \colon A \to \Sigma^2 I$.
Using the results of \cref{section:2}, we prove \cref{thm:sqz-general-case}.
and then deduce the prestable variant, i.e. \cref{thm:sqz-connective-variant}, using straightforward connectivity arguments.

\subsubsection{Acknowledgements}
I would like to thank my advisor, Tomer Schlank, for his invaluable guidance and support.
I would like to thank Tim Campion, Rune Haugseng, Maxime Ramzi, Jan Steinebrunner and Lior Yanovski for useful discussions regarding aspects of this work.
I would like to thank Janina C.~Letz and Elizabeth Tatum for their feedback on an early draft and Martin Gallauer for spotting some subtleties surrounding the Gray tensor product.
I would like to thank the Hausdorff Research Institute for Mathematics for their hospitality during the fall trimester program of 2022, funded by the Deutsche Forschungsgemeinschaft (DFG, German Research Foundation) under Germany's Excellence Strategy – EXC-2047/1 – 390685813.

\section{The split case}\label{section:2}

In this section we study the interplay between split square zero extensions and colax fixed points, which we shall now briefly recall.
Let $\calC$ be an \category{} and $T \colon \calC \to \calC$ an endofunctor.
The \textit{colax fixed points} of $T$ acting on $\calC$ is an \category{} $\Fixcolax_T(\calC)$ whose objects are pairs $(X,\varphi)$ where $X \in \calC$ and $\varphi \in \Map_\calC(X,T(X))$ and whose morphisms $(X,\varphi) \to (Y,\psi)$ are commutative squares:
\[\begin{tikzcd}
	X & Y \\
	{T(X)} & {T(Y)}
	\arrow["f", from=1-1, to=1-2]
	\arrow["\varphi"', from=1-1, to=2-1]
	\arrow["\psi", from=1-2, to=2-2]
	\arrow["{T(f)}"', from=2-1, to=2-2]
\end{tikzcd}\]
The main goal of this section is to prove \cref{cor:bifiber-seq-general-case} which, when specialized to the case
$\calA=\mcal{N}=\Sp$,
states that for any $V \in \Sp$ there is a canonical coreflective adjunction:
\[\begin{tikzcd}
	{\tild{i}_!\colon \LMod_{\Lambda (\Sigma^{-1} V)}(\Sp)} && { \Fixcolax_{V \otimes (-)}\left(\Sp\right) \colon \tild{i}^\ast}
	\arrow[""{name=0, anchor=center, inner sep=0}, shift left=2, hook, from=1-1, to=1-3]
	\arrow[""{name=1, anchor=center, inner sep=0}, shift left=2, from=1-3, to=1-1]
	\arrow["\dashv"{anchor=center, rotate=-90}, draw=none, from=0, to=1]
\end{tikzcd}\]
where $\Lambda(\Sigma^{-1}V) \coloneqq \bbS \ltimes (\Sigma^{-1}V)$.
We will deduce this from \cref{thm:fixcolax-categorifies-exterior} using an enriched variant of Barr-Beck-Lurie (\cref{lem:recollement-from-adjunction}).
\cref{thm:fixcolax-categorifies-exterior} and its proof constitute the technical backbone of this paper.
When spelled out for $\calA=\calN=\Sp$, \cref{thm:fixcolax-categorifies-exterior} says, informally speaking, that $\LMod_{\Lambda (\Sigma^{-1} V)}$ is the best approximation of $\Fixcolax_{V \otimes(-)}(\Sp)$ by an \category{} of modules over an algebra.

\begin{rem}
    Let us offer a precise definition of $\Fixcolax_T(\calC)$.
    Note that the pair $(\calC,T)$ determines (and is determined by) a unique functor $F_{(\calC,T)} \colon \rmB \bbN \too \CatI$, the unstraightning of which constitutes a cocartesian fibration 
    $\Un_{\rmB \bbN}(F_{(\calC,T)}) \to \rmB \bbN$. 
    One can define 
    $\Fixcolax_T(\calC)$ as the \category{} of sections:
    \[\Fixcolax_T(\calC)\coloneqq \Fun_{/\rmB \bbN}\left(\rmB \bbN,\Un_{\rmB \bbN}(F_{(\calC,T)})\right)\]
\end{rem}

As its name suggests 
$\Fixcolax_T(\calC)$ can be thought of as the colax limit of the functor 
$\rmB \bbN \to \CatI$ classified by $(\calC,T)$.
The $2$-categorical nature of $\Fixcolax_T(\calC)$ does not appear in the formulation of \cref{cor:bifiber-seq-general-case}, 
but will be essential for its proof.
In the first two subsections we lay some foundational context for the paper.
As such these subsections are highly technical and are best skipped at first reading.

\subsection{\texorpdfstring{$(\infty,2)$-}{Infinity 2-}categorical preliminaries}

The \category{} of \twocategories{} admits many equivalent descriptions but there seems to be no concensus on a single-valued definition.
For the purpose of this paper we define the term \textit{\hl{\twocategory{}}} to mean $\CatI$-enriched \category{}, 
where enrichment is taken in the sense of \cite{enriched}.  
We write $\hl{\twoCat}$ for the \category{} of \twocategories{}.
This model for $\twoCat$ was shown in \cite{RuneTwoCat} to be equivalent to complete $2$-fold Segal spaces \cite{BarwickThesis}. 
The latter was shown in  \cite{unicity} to be equivalent to complete Segal $\Theta_2$-spaces \cite{Rezk}. 
For a more detailed account on various models we refer the reader to \cite{RuneMonad}.

Restricting enrichment along the monoidal adjunction
$|-| \colon  \CatI \adj \calS$
gives rise to an adjunction:
\[\begin{tikzcd}
	{\mrm{include} \colon \CatI} && {\twoCat \colon \iota}
	\arrow[""{name=0, anchor=center, inner sep=0}, shift left=2, hook, from=1-1, to=1-3]
	\arrow[""{name=1, anchor=center, inner sep=0}, shift left=2, from=1-3, to=1-1]
	\arrow["\dashv"{anchor=center, rotate=-90}, draw=none, from=0, to=1]
\end{tikzcd}\]
The right adjoint sends an \twocategory{} $\frX \in \twoCat$ to its underlying \category{} $\iota \frX \in \CatI$.
For $\epsilon \in \{0,1\}$ we let $C_\epsilon \in \Cat_{(\infty,2)}$ denote $[\epsilon] \in \CatI$ considered as an $(\infty,2)$-category and $\partial C_1 \coloneqq  C_0 \amalg C_0 \in \twoCat$.
We write $\enHom_\frX(-,-) \colon (\iota \frX)^\op \times \iota \frX \to \Cat_\infty$
for the $\CatI$-enriched hom functor.
The cartesian product in $\twoCat$ preserves colimits in each variable, and thus we have an associated enriched hom:
\[\FUN(-,-) \colon \twoCat^\op \times \twoCat \to \twoCat\]
We write $\hl{\laxotimes} \colon \Cat_{(\infty,2)} \times \Cat_{(\infty,2)} \too \Cat_{(\infty,2)}$ 
for the \textit{\hl{Gray tensor product}} - a monoidal structure introduced initially for $2$-categories in \cite{formal}, and extended to \twocategories{} in \cite{maeharalax}.
We refer the reader to \cite[\S 3]{RuneMonad} for precise definitions as well as a brief summary of its basic properties.\footnote{We do not assume anything about the Gray tensor product beyond what is currently known. Indeed, $(1)$ and $(2)$ in \cite[Assumptions 3.2]{RuneMonad} are now theorems due to Maehara \cite{maeharalax} and $(3)$ will be entirely irrelevant for the purposes of this paper.}
The Gray tensor product defines a colimit preserving symmetric monoidal structure on $\twoCat$ with the property that the tensor product $C_1 \laxotimes C_1 \in \twoCat$ is given by the lax commuting square:
    \[\begin{tikzcd}
	\bullet & \bullet \\
	\bullet & \bullet
	\arrow[from=1-1, to=2-1]
	\arrow[from=1-1, to=1-2]
	\arrow[from=1-2, to=2-2]
	\arrow[from=2-1, to=2-2]
	\arrow[Rightarrow, from=1-2, to=2-1]
    \end{tikzcd}\]
We denote the internal hom adjoint to the Gray monoidal structure as:
\[\FUN(-,-)_\lax \colon \twoCat^\op \times \twoCat \too \twoCat \]
The Gray monoidal structure is \textit{not} symmetric. We denote its reverse by 
$\frX \hl{\colaxotimes} \frY \coloneqq  \frY \laxotimes \frX$ 
and the corresponding internal hom by:
\[\FUN(-,-)_\colax \colon \twoCat^\op \times \twoCat \too \twoCat \]

Let $\frX$ be an \twocategory{}. 
Given $X \in \frX$ and $J \in \CatI$, the \textit{\hl{cotensor}} of $X$ by $J$, when it exists, is an object $\hl{X^J} \in \frX$ representing the following presheaf:
\[ \Map_{\CatI}(J,\enHom_\frX(-,X)) \colon (\iota \frX)^\op \too \calS \qquad Y \longmapsto \Map_{\CatI}(J,\enHom_\frX(Y,X))\]

\begin{defn}
    We say that an \twocategory{} $\frX \in \twoCat$ is \textit{\hl{finitely bi-complete}} if it admits all finite cotensors and its underlying \category{} is finitely complete.
\end{defn}

\begin{notation}
    \begin{enumerate}
        \item We let $\frend \in \Cat_{(\infty,2)}$ denote the free walking endomorphism.
        \item We let $\frmnd \in \Cat_{(\infty,2)}$ denote the free walking monad. 
        \item We let $\fradj \in \Cat_{(\infty,2)}$ denote the free walking adjunction.
    \end{enumerate}
\end{notation}
\begin{defn}
    Let $\frX$ be an $(\infty,2)$-category.
    \begin{enumerate}
        \item 
        We define the \hl{\textit{fixed points}} functor, when it exists, to be the right adjoint of the adjunction:
        \[\mrm{const} \colon \frX \adj \FUN(\frend,\frX) \colon \hl{\FIX}\]
        \item 
        We define the \hl{\textit{colax fixed points}} functor, when it exists, to be the right adjoint in the adjunction:
	    \[ \mrm{const}:\frX \adj \FUN(\mfr{end},\frX)_\colax: \hl{\FIXcolax}\]
    \end{enumerate}
\end{defn}
Note that the natural inclusion $\FUN(\frend,\frX) \hookrightarrow \FUN(\frend,\frX)_\colax$ induces an equivalence on underlying \categories{}.
In particular, in both cases, objects correspond to pairs $(X,T)$ where $X \in \frX$ and $T \colon X \to X$ is an endomorphism.
The following lemma describe the morphisms between such pairs.

\begin{lem}\label{lem:mapping-spaces-in-end-lax}
    Suppose $(X,T)$ and $(X',T')$ are objects of  $\FUN(\frend,\frX)_\colax$.
    Then there is a canonical pullback square:
    \[\begin{tikzcd}
		{\Map_{\FUN(\mfr{end},\frX)_\colax}\left((X',T'),(X,T)\right)} & {\Ar\left(\enHom_\frX(X',X)\right)^{\simeq}} \\
		{\enHom_\frX(X',X)^{\simeq}} & {\enHom_\frX(X',X)^{\simeq} \times \enHom_\frX(X',X)^{\simeq}}
		\arrow[from=1-1, to=2-1]
		\arrow[from=1-2, to=2-2]
		\arrow["{\left((-)\circ T',T \circ(-)\right)}", ""{name=0, anchor=center, inner sep=0}, from=2-1, to=2-2]
		\arrow[from=1-1, to=1-2]
		\arrow["\lrcorner"{anchor=center, pos=0.125}, draw=none, from=1-1, to=0]
	\end{tikzcd}\]
\end{lem}
\begin{proof}
    Applying $\FUN(-,\frX)_\colax$ to the following pushout square
    \[\begin{tikzcd}
	{\partial C_1} & { C_1} \\
	{C_0} & {\mfr{end}}
	\arrow[from=1-2, to=2-2]
	\arrow[from=1-1, to=2-1]
	\arrow[from=2-1, to=2-2]
	\arrow[from=1-1, to=1-2]
	\arrow["\lrcorner"{anchor=center, pos=0.125, rotate=180}, draw=none, from=2-2, to=1-1]
    \end{tikzcd}\]
    and then hitting the result with $\Map_{\twoCat}(\partial C_1 \hookrightarrow C_1,-)$ gives a commutative cube
    \[\begin{tikzcd}[sep=0.8em, font=\footnotesize]
	& {\Map_{\twoCat} \left(C_1,\FUN(\frend,\frX)_\colax\right)} && {\Map_{\twoCat}\left(C_1,\FUN(C_1,\frX)_\colax\right)} \\
	{\Map_{\twoCat}\left(C_1,\frX\right)} && {\Map_{\twoCat}\left(C_1,\FUN(\partial C_1,\frX)_\colax\right)} \\
	& {\Map_{\twoCat} \left( \partial C_1,\FUN(\frend,\frX)_\colax\right)} && {\Map_{\twoCat}\left(\partial C_1,\FUN(C_1,\frX)_\colax \right)} \\
	{\Map_{\twoCat}\left(\partial C_1,\frX\right)} && {\Map_{\twoCat}\left(\partial C_1,\FUN(\partial C_1,\frX)_\colax\right)}
	\arrow[from=1-2, to=2-1]
	\arrow[from=1-4, to=2-3]
	\arrow[from=2-1, to=2-3]
	\arrow[from=1-2, to=1-4]
	\arrow[from=1-4, to=3-4]
	\arrow[from=2-3, to=4-3]
	\arrow[from=4-1, to=4-3]
	\arrow[from=3-2, to=3-4]
	\arrow[from=2-1, to=4-1]
	\arrow[from=1-2, to=3-2]
	\arrow[from=3-2, to=4-1]
	\arrow[from=3-4, to=4-3]
    \end{tikzcd}\]
    in which the top and bottom faces are cartesian.
    Taking fibers over 
    \begin{align*}
        \left([X' \xrightarrow{T'} X'],[X\xrightarrow{T}X]\right) \in \Map_{\twoCat}(\frend,\frX)^{\times 2}  & \simeq \Map_{\twoCat}(\frend \colaxotimes  \partial C_1 ,\frX) \\
        & \simeq \Map_{\twoCat}(\partial C_1 ,\FUN(\frend,\frX)_\colax) 
    \end{align*}
    yields the desired pullback square.
\end{proof}

\begin{lem}\label{lem: formula-for-colax-fixpoints}
    Let $\frX$ be a finitely complete $(\infty,2)$-category.
    For all $X \in \frX$ and $T \in \enEnd_{\frX}(X)$ 
    there is a canonical pullback square:
	\[\begin{tikzcd}
		{\FIXcolax(T\colon X \to X)} & {X^{\Delta^1}} \\
		X & {X^{\partial \Delta^1}}
		\arrow[from=1-2, to=2-2]
		\arrow[from=1-1, to=2-1]
		\arrow["{(\id,T)}", from=2-1, to=2-2]
		\arrow["\lrcorner"{anchor=center, pos=0.125}, draw=none, from=1-1, to=2-2]
		\arrow[from=1-1, to=1-2]
	\end{tikzcd}\]
\end{lem}
\begin{proof}
    Applying \cref{lem:mapping-spaces-in-end-lax} to the pair $(Y,\id_Y), (X,T)$ we get the following pullback square:
	\[\begin{tikzcd}
		{\Map_{\FUN(\mfr{end},\frX)_\colax}((Y,\id_Y),(X,T))} & {\Map_{\iota\frX}(Y,X^{\Delta^1})} \\
		{\Map_{\iota\frX}(Y,X)} & {\Map_{\iota\frX}(Y,X^{\partial \Delta^1})}
		\arrow[from=1-1, to=2-1]
		\arrow[from=1-2, to=2-2]
		\arrow[from=1-1, to=1-2]
		\arrow[""{name=0, anchor=center, inner sep=0}, from=2-1, to=2-2]
		\arrow["\lrcorner"{anchor=center, pos=0.125}, draw=none, from=1-1, to=0]
	\end{tikzcd}\]
	Consequently, the pullback satisfies the universal property of $\FIXcolax(T\colon X \to X)$.
\end{proof}

\begin{rem}
    Let $\frX$ be an \twocategory{} with a terminal object $\pt_\frX \in \iota \frX$. 
    In general it is not necessarily true that $\pt_\frX$ is \textit{bi-terminal}, i.e. it is not necessarily true that $\enHom_\frX(X,\pt_\frX) \simeq \pt \in \Cat_\infty$ for all $X \in \frX$.
    By uniqueness however, $\frX$ admits a bi-terminal object if and only if it admits a terminal object which satisfies the above property.
\end{rem}

\begin{defn}
    A \twocategory{} $\frX \in \Cat_{(\infty,2)}$ is called \textit{\hl{\good{}}} if it satisfies:
    \begin{enumerate}
        \item 
        $\frX$ is finitely bi-complete.
        \item 
        There exists an object $0_\frX \in \frX$, such that for all $X \in \frX$ we have equivalences
        \[\enHom_\frX(X,0_\frX)\simeq \pt \simeq \enHom_\frX(0_\frX,X)\]
        \item 
        For every $X \in \frX$ the terminal map $t_X \colon X \to 0_\frX$ is both left and right adjointable.
    \end{enumerate}
    Note that in this situation we have for every $X \in \frX$ a triple adjunction $i_X \dashv t_X \dashv i_X$ where $i_X \colon 0_\frX \to X$ denotes the initial map.
\end{defn}

\begin{example}\label{ex:modules-over-pointed-are-good}
    Let $\calB \in \Alg(\PrL)$ be a pointed presentably monoidal \category{}. 
    Presentable right $\mcal{B}$-modules form a \good{} \twocategory{}.
\end{example}

\begin{obs}\label{obs:terminal-objects}
    Let $\frX$ be a \good{} \twocategory{}.
    We claim that for all $Y \in \frX$ the composite 
    \[0_{Y \to X} \colon Y \xrightarrow{t_Y} \pt_\frX \xrightarrow{t_X^R} X\]
    is a zero object of $\enHom_\frX(Y,X)$.
    Indeed for any
    $f \colon Y \to X \in \enHom_\frX(Y,X)$ 
    we have
    \[\Map_{\enHom_\frX(Y,X)}(f,i_X \circ t_Y) \simeq \Map_{\enHom_\frX(Y,\pt_\frX)}(t_X \circ f, t_Y) \simeq \pt, \]
    and thus $i_X \circ t_Y \in \enHom_\frX(Y,X)$ is terminal.
    The same argument shows that $i_X \circ t_Y$ is initial.
\end{obs}

\begin{lem}\label{cor:fixlax-of-zero}
    Let $\frX \in \Cat_{(\infty,2)}$ be a \good{} \twocategory{}.
    Then there is a canonical equivalence:
	\[  \FIXcolax(0_{X \to X} \colon X \to X) \simeq X \]
\end{lem}
\begin{proof}
    By \cref{lem: formula-for-colax-fixpoints} we have for any $Y \in \frX$ a pullback square of \categories{}:
    \[\begin{tikzcd}
	{\enHom_\frX(Y,\Fixcolax(0_{X \to X} \colon X \to X))} && {\Ar(\enHom_\frX(Y,X))} \\
	{\enHom_\frX(Y,X)} && {\enHom_\frX(Y,X) \times \enHom_\frX(Y,X)}
	\arrow["{(\ev_0,\ev_1)}", from=1-3, to=2-3]
	\arrow["{(\Id,0_{X \to X} \circ (-))}"', from=2-1, to=2-3]
	\arrow[from=1-1, to=2-1]
	\arrow[from=1-1, to=1-3]
    \end{tikzcd}\]
    It thus suffices by Yoneda to show that the functor
    $0_{X \to X} \circ (-) \colon \enHom_\frX(Y,X) \to \enHom_\frX(Y,X)$
    maps all objects to the terminal object.
    This is a consequence of \cref{obs:terminal-objects} since for all $f \colon Y \to X$ we have
    $0_{X \to X} \circ f \simeq 0_{Y \to X}$.
\end{proof}

\begin{const}\label{defn:colaxfix-pointed}
    Let $\frX$ be a \twocategory{} and let $X \in \frX$. 
    As in \cite[Corollary 8.9.]{RuneMonad}, we have
    an equivalence of \categories{}
    \[\{X\} \times_\frX \FUN(\frend,\frX)_{\colax} \simeq \enEnd_\frX(X).\]
    In particular we obtain a functor $\enEnd_\frX(X) \to \FUN(\frend,\frX)_\colax$.
    Suppose $\frX$ is \good{} so that $\enEnd_\frX(X)$ is pointed by $0_{X \to X} \colon X \to X$.
    By \cref{cor:fixlax-of-zero} the composite
    \[\enEnd_\frX(X) \too \FUN(\frend,\frX)_\colax \xrightarrow{\FIXcolax} \frX,\]
    lifts canonically to $(\frX_{/X})_{\id_X/}$.
    We denote the corresponding lift by
    \[\hl{\Fixcolax_{(-)}(X)} \colon \enEnd_\frX(X) \too (\frX_{X/})_{/\id_X}.\]
\end{const}

    Although the main results of this section concern the \twocategory{} $\enRMod_\calA(\PrLst)$, 
    some of the proofs will require a slightly larger \twocategory{}. 
    There are two reasons for this. 
    Firstly, we will often have to consider right adjoints, or even worse, composites of left and right adjoints which in general are neither. 
    Secondly, the right adjoint of an $\calA$-linear functor often fails to be $\calA$-linear.
    It will always be $\calA$-lax-linear however (see \cite[Corollary 3]{RuneLax}).
    We will address these issues by working in a larger
    \twocategory{} in which $\enRMod_\calA(\PrLst)$ embeds.

    We denote by $\Op_\infty$ the \category{} of $\infty$-operads.
    We write $\Assoc \in \Op_\infty$ for the associative operad
    and $\RM$
    for the right
    module operad as defined in \cite[Definition 4.1.1.1 and Variant 4.2.1.36]{HA}).
    Recall that $\Assoc$ is equivalent to the $\bbE_1$-operad,
    hence to unify terminology we will often refer to $\Assoc$-algebras as $\bbE_1$-algebras.
    Given a monoidal \category{} $\calC$ we write ${\calC^\otimes} \to \Assoc$ for its underlying $\infty$-operad.\footnote{If one thinks of $\calC$ as a functor $\Assoc \to \CatI$ satisfying the Segal conditions, $\calC^\otimes \to \Assoc$ is simply the unstraightning.}
    We write ${\Alg(\calC)} \coloneqq  \Alg_{\Assoc}(\calC^\otimes)$
    and similarly 
    ${\RMod(\calC)} \coloneqq  \Alg_{\RM/\Assoc}(\calC^\otimes)$.

\begin{notation}\label{obs:enrichment-slice-operads}
    Given an $\infty$-operad $\calO$ we write $\Cat^\xint_{\infty/\calO^\otimes}$ for the subcategory of $\Cat_{\infty/\calO^\otimes}$ whose objects are functors $\calQ^\otimes \to \calO^\otimes$ which admit cocartesian lifts of inert edges and whose morphisms are functors preserving these cocartesian lifts. 
    Note that $\Op_{\infty/ \calO^\otimes}$ is a full subcategory of $\Cat^\xint_{\infty/\calO^\otimes}$ hence is canonically $\CatI$-enriched. 
    We write $\enOpover{\calO^\otimes}$ for the corresponding \twocategory{}.
    Pulling back along a morphism of $\infty$-operads
    $\calO^\otimes \to \calP^\otimes$ defines a $\CatI$-enriched functor $\enOpover{\calO^\otimes} \too \enOpover{\calP^\otimes}$.
\end{notation}

By \cite[Proposition 2.4.2.5]{HA} we may identify ${\Alg(\CatI)}\coloneqq \Alg_{\Assoc}(\CatI^\times)$ with the replete subcategory of $\Op_{\infty/\Assoc}$ 
whose objects are cocartesian fibrations and whose morphisms preserve \textit{all} cocartesian edges.
Similarly we can identify
${\RMod(\CatI)} \coloneqq  \Alg_{\RM}(\CatI^\times)$
with the corresponding subcategory of $\Op_{\infty/\RM}$. 
We may thus promote $\Alg(\CatI)$ and $\RMod(\CatI)$ to \twocategories{} which we denote by $\enAlg(\CatI)$ and $\enRMod(\CatI)$ respectively.
The \twocategory{} of right modules over a fixed monoidal \category{} $\calC \in \Alg(\CatI)$ is defined as 
${\enRMod_\calC(\CatI)}\coloneqq  \{\calC\} \times_{\enAlg(\CatI)} \enRMod(\CatI)$.
Given $\calU, \calV \in \RMod_\calC(\CatI)$
we write ${\Fun_\calC(\calU,\calV)}\coloneqq \enHom_{\enRMod_\calC(\CatI)}(\calU,\calV)$ for the \category{} of $\calC$-linear functors.

Let ${\enAlg^\lax(\CatI)} \subseteq \enOpover{\Assoc}$ denote the full sub-\twocategory{} spanned by the cocartesian fibrations of $\infty$-operads and similarly for ${\enRMod^\lax(\CatI)} \subseteq \enOpover{\RM}$.
Pulling back along the inclusion $\Assoc \hookrightarrow \RM$ defines a $(\infty,2)$-functor:
 \[ \forget \colon \enRMod^\lax(\CatI) \too \enAlg^\lax(\CatI)\]

\begin{defn}
    Given a monoidal \category{} $\calC \in \Alg(\CatI) \subseteq \enAlg^\lax(\CatI)$ we define the \twocategory{} of \hl{\textit{$\calC$-modules and $\calC$-lax-linear functors}} by:
    \[ \hl{\enRMod_\calC^\lax(\CatI)} \coloneqq  \left\{\calC\right\} \times_{\enAlg^\lax(\CatI)} \enRMod^\lax(\CatI) \quad \in \twoCat\]
    Given $\calU, \calV \in \RMod_\calC(\CatI)$
    we write $\Fun^\lax_\calC(\calU,\calV) \coloneqq \enHom_{\enRMod^\lax_\calC(\CatI)}(\calU,\calV)$ for the \category{} of $\calC$-lax-linear functors.
    Note that the natural functor  $\enRMod_\calC(\CatI) \to \enRMod^\lax_\calC(\CatI)$
    induces a fully faithful embedding:
    \[ \Fun_\calC(\calU,\calV) \hookrightarrow \Fun^\lax_\calC(\calU,\calV)\]
    A $\calC$-lax-linear functor
    $F \colon \calU \to \calV$ lies in the essential image if and only if it is $\calC$-linear, i.e. for every $U \in \calU$ and 
    $C \in \calC$ the lax linear structure map $F(U) \otimes C \iso F(U \otimes C)$ is an equivalence.
\end{defn}

\begin{obs} 
    Let $\calC \in \Alg(\CatI)$ and let $\calU,\calV \in \RMod_\calC(\CatI)$ be right $\calC$-modules corresponding to cocartesian fibrations $(\calC,\calU)^\otimes \to \RM$, and $(\calC,\calV)^\otimes\to \RM$ respectively. 
    We have a canonical equivalence:
    \begin{align*}
        \Fun^\lax_\calC(\calU,\calV) &\simeq \{\Id_{\calC^\otimes}\} \times_{\enHom_{\enOpover{\Assoc}}(\calC^\otimes,\calC^\otimes)} \enHom_{\enOpover{\RM}}\left((\calC,\calU)^\otimes,(\calC,\calU)^\otimes\right) \\
        & \simeq \{\Id_{\calC^\otimes}\} \times_{\Fun_{\Cat^\xint_{\infty/\Assoc}}(\calC^\otimes,\calC^\otimes)} \Fun_{\Cat^\xint_{\infty/\RM}}\left((\calC,\calU)^\otimes,(\calC,\calU)^\otimes\right)
    \end{align*}
    Passing to the fiber over $\mfr{m} \in \RM$ defines a $\CatI$-enriched functor $\enRMod^\lax_\calC(\CatI) \to \Cat_\infty$.
    Using the description of the hom categories above we see that on the corresponding hom categories it induces a conservative, limit preserving functor
    $\Fun^\lax_\calC(\calU,\calV) \to \Fun(\calU,\calV)$.
\end{obs}

The forgetful functor $\PrLst \to \CatI$ is faithful, replete, and lax monoidal. 
Consequently for $\calA \in \Alg(\PrLst)$ we may identify $\RMod_{\calA}(\PrLst)$ with the replete subcategory of $\RMod_{\calA}(\CatI)$ spanned by: 
\begin{itemize}
    \item 
    \textbf{Objects} - stable presentable right $\calA$-modules.\footnote{A \textit{presentable right $\calA$-module} is an $\calA$-module $\mcal{U}$ whose underlying \category{} is stable and presentable and whose action functor $\otimes \colon \mcal{U} \times \calA \to \mcal{U}$ preserves colimits in both variables.}
    \item
    \textbf{Morphisms} - colimit preserving $\calA$-linear functors.
\end{itemize}
We let $\enRMod_\calA(\PrLst)$ denote the corresponding sub-$(\infty,2)$-category of $\enRMod(\CatI)$.
Given $\mcal{U}, \mcal{V} \in \enRMod_\calA(\xPrL{}{\st})$ we write $\Fun^\mrm{L}_\calA(\mcal{U},\mcal{V}) \coloneqq \enHom_{\enRMod_\calA(\xPrL{}{\st})}(\mcal{U},\mcal{V})$ for the corresponding enriched hom. 
Similarly for $\mcal{U} \in \enRMod_\calA(\PrLst)$ we write $\enEnd_\calA^\mrm{L}(\mcal{U}) \coloneqq \enHom_{\enRMod_\calA(\PrLst)}(\mcal{U},\mcal{U})$.

\begin{obs}\label{obs:finite-limits-of-linear-left-adjoints}
    Let $\calA \in \Alg(\PrLst)$ and let $\calU,\calV \in \enRMod_\calA(\PrLst)$.
    The composite functor
    \[\enRMod_\calA(\PrLst) \too \enRMod_\calA(\CatI) \too \enRMod^\lax_\calA(\CatI)\]
    induces a fully faithful embedding 
    $\Fun^\mrm{L}_\calA(\calU,\calV) \hookrightarrow \Fun^\lax_\calA(\calU,\calV)$.
    We claim that this embedding creates and preserves finite limits.
    This follows by observing that the forgetful functors 
    \[\Fun_\calA^\lax(\calU,\calV) \to \Fun(\calU,\calV) \quad \text{ and } \quad\Fun^\mrm{L}_\calA(\calU,\calV) \to \Fun(\calU,\calV)\]
    are conservative and preserve finite limits (for the latter functor this holds by stability).
\end{obs}

\begin{notation}
    For $\frX \in \twoCat$ we let $\hl{\frX^\ladj}$ and $\hl{\frX^\radj}$ denote the sub-\twocategories{} of $\frX$ spanned by the left and right adjoints respectively.
\end{notation}

\begin{obs}\label{obs:passing-to-adjoint}
    By \cite[Corollary C]{RuneLax}
    , passing to left adjoints gives rise to a commutative square of \twocategories{}:
    \[\begin{tikzcd}
    	{\RModlax(\CatI)^\radj} & {\left(\RModlax(\CatI)^\ladj\right)^{(1,2)-\op}} \\
    	{\enAlg^\lax(\CatI)^\radj} & {\left(\enAlg^\lax(\CatI)^\ladj\right)^{(1,2)-\op}}
    	\arrow["\simeq", from=1-1, to=1-2]
    	\arrow["\forget"', from=1-1, to=2-1]
    	\arrow["\forget"', from=1-2, to=2-2]
    	\arrow["\simeq", from=2-1, to=2-2]
    \end{tikzcd}\]
    where both horizontal functors are equivalences.
    Since the identity functor is self adjoint we may pass to vertical fibers at $\calC \in \enAlg^\lax(\CatI)$ to obtain an equivalence of \twocategories{}:
    \[\RModlax_\calC(\CatI)^\radj \iso \left(\RModlax_\calC(\CatI)^\ladj\right)^{(1,2)-\op}\]
\end{obs}

\subsection{Colax fixed points in the stable setting}

Let $\calA$ be a stable presentably monoidal \category{}.
Presentable right $\calA$-modules form a \good{} \twocategory{} $\enRMod_\calA(\PrLst)$ (see \cref{ex:modules-over-pointed-are-good}).
Given $\mcal{N} \in \enRMod_\calA(\PrLst)$, \cref{defn:colaxfix-pointed} supplies a functor:
\[\Fixcolax_{(-)}(\mcal{N}) \colon \mbf{End}_{\enRMod_{\calA}(\PrLst)}(\mcal{N}) \too \left(\enRMod_{\calA}(\PrLst)_{\mcal{N}/}\right)_{/\id_{\mcal{N}}}\]
In this subsection we prove \cref{prop:fixcolax-stable-setting} which shows that due to stability this functor is much better behaved than one might initially expect.

Throughout this section we will encounter left adjoint functors of the form
$q(\infty) \to \lim_{i \in I} q(i)$
where $q \colon I^{\large{\triangleleft}} \to \PrL$ is some diagram.
Using \cite[Theorem 5.5]{AdjDescent} we may compute the right adjoint of such a functor at an object by taking a limit in $q(\infty)$ of a certain canonical diagram indexed by $I^\op$ which is assembled from all the adjunctions. 
We will use the term \textit{"adjoint descent"} to refer to this technique.

\begin{notation}
    In this section we use the following notation.
    \begin{itemize}
        \item 
        We fix a stable presentably monoidal \category{} $\calA \in \Alg(\PrLst)$.
        \item 
        We fix a presentable right $\calA$-module $\mcal{N} \in \RMod_\calA(\PrLst)$.
        \item 
        We fix a colimit preserving $\calA$-linear endofunctor $T \colon \mcal{N} \to \mcal{N} \in \Fun^{\mrm{L}}_\calA(\mcal{N},\mcal{M})$.
\end{itemize}
For the next couple of proofs we will also need the following more specific notation.
\begin{itemize}
        \item 
        We write $\ev_\epsilon \colon \mcal{N}^{\Delta^1} \to \mcal{N}$ for the evaluation at $\epsilon \in \{0,1\}$ and $\ev^L_\epsilon \colon \mcal{N} \to \mcal{N}^{\Delta^1}$ for its left adjoint.
        \item 
        We write
        $\can \colon \ev_0 \to \ev_1$
        for the tautological natural transformation
        and
        $\can^L \colon \ev_1^L \to \ev_0^L$
        for its left mate.
        \item 
        Given an ($\calA$-linear) natural transformation 
        $\alpha \colon \Id \to T$
        we let
        $j_{\alpha !} \colon \mcal{N} \too \mcal{N}^{\Delta^1}$ 
        denote the $\calA$-linear functor defined by $X \longmapsto  \left(X \xrightarrow{\alpha} T(X)\right)$.
        In the case $\alpha =0$ we write $j_!\coloneqq j_{0,!}$. 
    \end{itemize}
\end{notation}

\begin{lem}\label{lem:right-adjoint-from-arrow-category}
    Let $\alpha \colon \Id_\mcal{N} \to T$ be an $\calA$-linear natural transformation.
    The right adjoint $j^{\ast}_\alpha \colon \mcal{N}^{\Delta^1} \to \mcal{N}$
	sits in a pullback square of $\mcal{A}$-lax-linear functors:
	\[\begin{tikzcd}
		{j_\alpha^\ast} && \ev_1\\
		{\ev_0 \oplus T^R\ev_1} && {\ev_1 \oplus \ev_1}
		\arrow["\Delta", from=1-3, to=2-3]
		\arrow[from=1-1, to=1-3]
		\arrow[from=1-1, to=2-1]
		\arrow[""{name=0, anchor=center, inner sep=0}, "{\can \oplus \alpha^R (\ev_1)}"', from=2-1, to=2-3]
		\arrow["\lrcorner"{anchor=center, pos=0.125}, draw=none, from=1-1, to=0]
	\end{tikzcd}\]
	where $\alpha^R \colon T^R \to \Id_\mcal{N}$ denotes the right mate of $\alpha$.
\end{lem}
\begin{proof}
	Consider the following cube in $\Fun^\lax_\calA(\mcal{N}^{\Delta^1},\mcal{N})$:
	\[\begin{tikzcd}
	0 &&& 0 \\
	&& {\ev_0} &&& {\ev_0} \\
	{\ev_0} &&& {\ev_1} \\
	&& {\ev_0} &&& {\ev_1}
	\arrow[from=1-1, to=3-1]
	\arrow[from=1-1, to=1-4]
	\arrow[from=1-4, to=2-6]
	\arrow["{=}"{description}, from=3-1, to=4-3]
	\arrow["{\can}"{description, pos=0.4}, from=3-1, to=3-4]
	\arrow["{=}"{description}, from=3-4, to=4-6]
	\arrow["{\can}"{description}, from=4-3, to=4-6]
	\arrow["{=}"{description, pos=0.3}, from=2-3, to=4-3]
	\arrow[from=1-1, to=2-3]
	\arrow[from=1-4, to=3-4]
	\arrow["{\can}"{description}, from=2-6, to=4-6]
	\arrow["{=}"{description}, from=2-3, to=2-6]
    \end{tikzcd}\]
    Treating the vertical arrows as $\calA$-lax-linear functors $\mcal{N}^{\Delta^1} \to \mcal{N}^{\Delta^1}$ 
    we may interpret this cube as a pushout square in $\Fun^\lax_\calA(\mcal{N}^{\Delta^1},\mcal{N}^{\Delta^1})$.
    Precomposing with $j_{\alpha,!}$ then yields the following pushout square:
	\[\begin{tikzcd}
		{\ev_1^L\ev_0 j_{\alpha,!}} && {\ev_1^L\ev_1 j_{\alpha,!}} \\
		\\
		{\ev_0^L\ev_0 j_{\alpha,!}} && {j_{\alpha,!}}
		\arrow["{c_{\ev_1}(j_{\alpha,!})}"{description}, from=1-3, to=3-3]
		\arrow["{c_{\ev_0}(j_{\alpha,!})}"{description}, from=3-1, to=3-3]
		\arrow["{\can^L(\ev_0j_{\alpha,!})}"{description}, from=1-1, to=3-1]
		\arrow["{\ev_1^L(\can(j_{\alpha,!}))}"{description}, from=1-1, to=1-3]
		\arrow["\lrcorner"{anchor=center, pos=0.125, rotate=180}, draw=none, from=3-3, to=1-1]
	\end{tikzcd}\]
	One readily checks that
	$\ev_0(\can(j_{\alpha,!})):\ev_0j_{\alpha,!}\longrightarrow \ev_1 j_{\alpha,!}$ 
	is canonically equivalent to
	$\alpha: \Id \longrightarrow T$. 
	Under this identification, passing to right adjoints in the above pushout square gives the following square:
	\[\begin{tikzcd}
		{j^\ast_\alpha} & {T^R \ev_1} \\
		{\ev_0} & {\ev_1}
		\arrow["\lrcorner"{anchor=center, pos=0.125}, draw=none, from=1-1, to=2-2]
		\arrow[from=1-1, to=1-2]
		\arrow["{\alpha^R (\ev_1)}", from=1-2, to=2-2]
		\arrow["{\can}"', from=2-1, to=2-2]
		\arrow[from=1-1, to=2-1]
	\end{tikzcd}\]
	which is cartesian by \cref{obs:passing-to-adjoint}. 
	The required pullback square is now easily extracted.
\end{proof}

\begin{notation}\label{notation:notation-for-colaxfix}
    By \cref{lem: formula-for-colax-fixpoints} the lax fixed points $\Fixcolax_T(\mcal{N})$ sits in a pullback square:
	\label{diag:lax-fix-pullback-linear}
	\[\begin{tikzcd}
		{\Fixcolax_T(\mcal{N})} && {\mcal{N}^{\Delta^1}} \\
		{\mcal{N}} && {\mcal{N}^{\partial\Delta^1}}
		\arrow["\pi_!",from=1-1, to=1-3]
		\arrow["{(\ev_0,\ev_1)}",from=1-3, to=2-3]
		\arrow["q_!"', from=1-1, to=2-1]
		\arrow[""{name=0, anchor=center, inner sep=0}, "{(\id,T)}"', from=2-1, to=2-3]
		\arrow["\lrcorner"{anchor=center, pos=0.125}, draw=none, from=1-1, to=0]
	\end{tikzcd}\]
	where $q_!$ and $\pi_!$ are uniquely determined and are given by:
    \[\hl{q_!} \colon (X,\varphi) \longmapsto X \quad \quad \hl{\pi_!} \colon (X,\varphi) \longmapsto \varphi\]
    We write $\hl{\gamma_\can} : q_! \to T q_!$ for the tautological natural transformation.
\end{notation}

\begin{const}\label{const:defn-of-section-for-colaxfix}
	An $\calA$-linear natural transformation $\alpha : \Id \to T$ 
	gives rise to a commutative square in $\RMod_\calA(\PrLst)$:
	\[\begin{tikzcd}
		{\mcal{N}} && {\mcal{N}^{\Delta^1}} \\
		{\mcal{N}} && {\mcal{N}^{\partial \Delta^1}}
		\arrow["{(\Id \xrightarrow{\alpha} T)}", from=1-1, to=1-3]
		\arrow[from=1-3, to=2-3]	
		\arrow["{(\Id,T)}", from=2-1, to=2-3]
		\arrow["\Id"', from=1-1, to=2-1]
	\end{tikzcd}\]
    which in turns gives rise to a section of the forgetful functor
    $q_! \colon \Fixcolax_T(\mcal{N})\to \mcal{N}$.
    In particular the natural transformation 
    $0 \colon \Id \to T$ 
    gives a section $i_! \colon \mcal{N} \to \Fixcolax_T(\mcal{N})$ and we denote the corresponding adjunction by:
    \[\hl{i_!} \colon \mcal{N} \adj \Fixcolax_T(\mcal{N}) \colon \hl{i^\ast}\]
\end{const}

\begin{prop}\label{prop:cofib-seq-laxfix}
	There is a fiber sequence in $\Fun^\lax_\calA(\Fixcolax_T(\mcal{N}),\mcal{N})$ of the form:
	\[ i^{\ast} \xrightarrow{i^{\ast}(u_q)} i^{\ast}q^{\ast}q_! \simeq q_! \overset{\gamma_\can}{\longrightarrow} T q_!\]
\end{prop}
\begin{proof}
	Consider the following natural diagram in $\Fun^\lax_\calA(\Fixcolax_T(\mcal{N}),\mcal{N})$:
	\[\begin{tikzcd}
		{i^{\ast}} & {i^{\ast}\pi^{\ast}\pi_!} & {j^{\ast} \pi_!} & {\ev_1\pi_!} \\
		{i^{\ast}q^{\ast} q_!} & {i^{\ast}q^{\ast}q_! \oplus q^{\ast}T^RT q_!} & {(\ev_0 \oplus T^R \ev_1 )\pi_!} & {(\ev_1 \oplus \ev_1)\pi_!}
		\arrow[from=1-1, to=2-1]
		\arrow[from=1-1, to=1-2]
		\arrow[from=1-2, to=2-2]
		\arrow[from=2-1, to=2-2]
		\arrow["\simeq", from=2-2, to=2-3]
		\arrow["\simeq", from=1-2, to=1-3]
		\arrow[from=2-3, to=2-4]
		\arrow[from=1-3, to=2-3]
		\arrow[from=1-4, to=2-4]
		\arrow[from=1-3, to=1-4]
	\end{tikzcd}\]
	where the left most square is obtained by letting the right adjoint 
	$i^{\ast}$ act on the adjoint descent diagram of the pullback square (\hyperref[diag:lax-fix-pullback-linear]{$\star$}), the right most square is obtained by precomposing \cref{lem:right-adjoint-from-arrow-category} with $\pi_!$, and the middle square is induced from the equivalences:
	\[j_! \simeq \pi_!  i_! \qquad q_!  i_! \simeq \Id_{\mcal{N}} \qquad q_! \simeq \ev_0   \pi_!  \qquad  T q_! \simeq  \ev_1  \pi_! \]
    The left and right most squares are cartesian and therefore so is the outer rectangle.
	The cartesian outer rectangle sits as the top left square in the following diagram:
	\[\begin{tikzcd}
		{i^{\ast}} & {\ev_1\pi_!} & {Tq_!} \\
		{i^{\ast}q^{\ast} q_!} & {(\ev_1 \oplus \ev_1)\pi_!} \\
		{q_!} && {Tq_! \oplus Tq_!}
		\arrow["{i^{\ast}(u_q)}"', from=1-1, to=2-1]
		\arrow[from=1-2, to=2-2]
		\arrow[from=2-1, to=2-2]
		\arrow[from=1-1, to=1-2]
		\arrow["\simeq", from=1-2, to=1-3]
		\arrow["\simeq"', from=2-1, to=3-1]
		\arrow["\simeq", from=2-2, to=3-3]
		\arrow["{\gamma_\can \oplus 0}", from=3-1, to=3-3]
		\arrow["\Delta", from=1-3, to=3-3]
		\arrow["\lrcorner"{anchor=center, pos=0.125}, draw=none, from=1-1, to=2-2]
	\end{tikzcd}\]
    We see that the outer rectangle is a pullback square.
    The desired fiber sequence is now easily extracted.
\end{proof}

\begin{prop}\label{prop:fixcolax-stable-setting}
    Colax fixed points defines a functor
    \[\Fixcolax_{(-)}(\mcal{N}) \colon \enEnd^\mrm{L}_{\calA}(\mcal{N}) \too \left(\enRMod_{\calA}(\PrLst)^{\ladj}_{\mcal{N}/}\right)_{/\Id_\mcal{N}}, \qquad T \longmapsto \left( \mcal{N} \xrightarrow{i_!}  \Fixcolax_T(\mcal{N}) \xrightarrow{q_!} \mcal{N} \right)\]
    The internal right adjoint of $i_! \colon \mcal{N} \to \Fixcolax_T(\mcal{N})$ is given by 
    \[i^\ast \colon \Fixcolax_T(\mcal{N}) \too \mcal{N}, \qquad \left(X,\varphi\right) \longmapsto \fib\left(\varphi \colon X \to T(X)\right)\]
\end{prop}
\begin{proof}
    $\enEnd_{\calA}(\mcal{N})$ 
    can be identified with the full subcategory of 
    $\Fun^\lax_\calA(\mcal{N},\mcal{N})$
    spanned by the colimit preserving (strongly) $\calA$-linear functors. 
    The results of the previous section applied to 
    $\frX \coloneqq  \RModlax_\calA(\CatI)$ and $X\coloneqq \mcal{N}$
    supply a functor:
    \[\Fixcolax_{(-)}(\mcal{N}) \colon \enEnd_{\calA}(\mcal{N}) \too (\RModlax_\calA(\CatI)_{\mcal{N}/})_{/\Id_\mcal{N}}\]
    We must show that this functor lands in
    $\left(\enRMod_\calA(\PrLst)^\ladj_{\mcal{N}/}\right)_{\Id_\mcal{N}/} \subseteq (\RModlax_\calA(\CatI)_{\mcal{N}/})_{/\Id_\mcal{N}}$.
    The inclusion functor
    $\enRMod_\calA(\PrLst) \to \RModlax_\calA(\CatI)$ 
    preserves limits and cotensors
    and thus by \cref{lem: formula-for-colax-fixpoints} we have
    $\FIXcolax(T \colon \mcal{N} \to \mcal{N}) \in \enRMod_\calA(\PrLst)$.
    To conclude we must show that the functor 
    $i_! \colon \mcal{N} \to \Fixcolax_T(\mcal{N})$ 
    is internally left adjoint, i.e. that its right adjoint $i^\ast$ is colimit preserving and $\calA$-linear.
    By \cref{prop:cofib-seq-laxfix} we have $i^\ast \simeq \fib(q_! \to Tq_!)$ where $q_!$ and $Tq_!$ are clearly colimit preserving and $\calA$-linear, and thus the claim follows from \cref{obs:finite-limits-of-linear-left-adjoints}
\end{proof}

\begin{const}\label{obs:shift-functors}
    Consider the following diagram in $\RMod_\calA(\PrLst)$:
    \[\begin{tikzcd}
	{\Fixcolax_T(\mcal{N})} &&& {\mcal{N}^{\Delta^1}} \\
	{\mcal{N}} &&& {\mcal{N}^{\partial \Delta^1}}
	\arrow["{(Tq_! \xrightarrow{T(\gamma_\can)} T^2q_!)}", from=1-1, to=1-4]
	\arrow["{Tq_!}"', from=1-1, to=2-1]
	\arrow["{(\Id,T)}", from=2-1, to=2-4]
	\arrow["{(\ev_0,\ev_1)}", from=1-4, to=2-4]
    \end{tikzcd}\]
    The canonical map to the pullback defines a shift functor: 
    \[\sigma \colon \Fixcolax_T(\mcal{N}) \to \Fixcolax_T(\mcal{N}), \qquad \sigma \colon (X,\varphi) \mapsto (T(X),T(\varphi))\]
    There is also a canonical natural transformation $\gamma_\can \colon \id \to \sigma$ which evaluates on $(X,\varphi)$ to the tautologically commuting square:
    \[\begin{tikzcd}
	X && {T(X)} \\
	{T(X)} && {T^2(X)}
	\arrow["\varphi", from=1-1, to=1-3]
	\arrow["\varphi"', from=1-1, to=2-1]
	\arrow["{T(\varphi)}", from=1-3, to=2-3]
	\arrow["{T(\varphi)}", from=2-1, to=2-3]
    \end{tikzcd}\]
    The colimit $\sigma^\infty\coloneqq \colim_n \sigma^n \colon \Fixcolax_T(\mcal{N}) \to \Fixcolax_T(\mcal{N})$ 
    is an idempotent functor since $\sigma$ preserves colimits.
    Observe now that $(X,\varphi) \simeq \sigma^\infty(X,\varphi)$ if and only if $(X,\varphi)\simeq \sigma(X,\varphi)$ if and only if $\varphi \colon X \to T(X)$ is an equivalence.
    We conclude that $\sigma^\infty$ reflects onto the full subcategory $\Fix_T(\mcal{N}) \subseteq \Fixcolax_T(\mcal{N})$.
\end{const}

\subsection{Exterior algebras as monadic approximations}

In this subsection we finally state and prove \cref{thm:fixcolax-categorifies-exterior}, which establishes the promised relation between split square zero extensions and colax fixed points.

\begin{defn}
    Let $\calC$ be a stable presentably monoidal \category{}. 
    We define the \hl{\textit{augmentation ideal}} functor as:
    \[ \hl{\overline{(-)}} \colon \Alg^\aug(\calC) \too \calC, \qquad \left(\unit \to R \to \unit \right) \longmapsto \hl{\overline{R}} \coloneqq  \cofib(\unit \to R) \simeq \fib(R \to \unit)\]
    By \cite[Theorem 7.3.4.7]{HA}, this functor induces an equivalence 
    $\overline{(-)} \colon \Sp(\Alg^{\aug}(\calC)) \simeq \Sp(\calC) \simeq \calC$.
    We define the \hl{\textit{split square zero extension}} functor as the composite:
    \[\hl{\Lambda(-)} \coloneqq \unit_\calC \ltimes (-) \colon \calC \simeq \Sp(\Alg^\aug(\calC)) \xrightarrow{\Omega^\infty} \Alg^\aug(\calC)\] 
\end{defn}

\begin{rem}
    Our choice of the abbreviation $\Lambda(-)$ for the split extension functor $\unit_\calC \ltimes (-)$ is motivated by the fact that split square zero extensions can be thought of as generalizations of exterior algebras. 
\end{rem}

\begin{const}\label{const:monad-from-adjunction}
    Given $\frX \in \Cat_{(\infty,2)}$ and $X \in \frX$ we construct a functor
    \[ \mnd(-) \colon \left(\frX^\ladj_{X/}\right)_{/\id_X} \too \Alg^\aug(\enEnd_\frX(X)), \qquad \left(f_! \colon X \adj Y \colon f^\ast \right) \longmapsto \mnd(f_! \dashv f^\ast) \coloneqq  f^\ast f_!\]
    Consider the commuting square of $(\infty,2)$-categories
    \[\begin{tikzcd}
	{\Delta^{\{0\}}} & \frmnd \\
	{\Delta^1} & \fradj
	\arrow[from=2-1, to=2-2]
	\arrow[from=1-1, to=2-1]
	\arrow[from=1-2, to=2-2]
	\arrow[from=1-1, to=1-2]
    \end{tikzcd}\]
    where the bottom map chooses the left adjoint. Hitting this square with $\FUN(-,\frX)_\colax$ produces the following square:
    \[\begin{tikzcd}
    {\FUN(\fradj,\frX)_\colax} & {\FUN(\frmnd,\frX)_\colax} \\
    {\FUN(\Delta^1,\frX)_\colax} & \frX
    \arrow[from=1-1, to=2-1]
    \arrow[from=1-1, to=1-2]
    \arrow["{\ev_0}", from=2-1, to=2-2]
    \arrow[from=1-2, to=2-2]
    \end{tikzcd}\]
    By \cite[Corollary 4.8]{RuneMonad}, the left vertical map induces an equivalence $\FUN(\fradj,\frX)_\colax \iso \FUN(\Delta^1,\frX)_\ladj$ where $\FUN(\Delta^1,\frX)_\ladj \subseteq \FUN(\Delta^1,\frX)$ denotes the full sub $(\infty,2)$-category on the left adjoint morphisms.
    Passing to the fiber at $X \in \frX$ and using the equivalence of \cite[Corollary 8.9]{RuneMonad} we get a functor
    \[\frX^\ladj_{X/} \to \FUN(\frmnd,\frX)_{\colax,X} \iso \Alg(\enEnd_\frX(X)). \]
    Slicing over $\id_X \colon X \to X \in \frX^\ladj_{X/}$ gives the desired functor.
\end{const}

\begin{prop}\label{thm:fixcolax-categorifies-exterior}
    There exists a canonically commuting square of \categories{}:
    \[\begin{tikzcd}
	{\enEnd^\mrm{L}_\calA(\mcal{N})} &&& {\left(\enRMod_{\calA}(\PrLst)^{\ladj}_{\mcal{N}/}\right)_{/\Id_\mcal{N}}} \\
	{\enEnd^\mrm{L}_\calA(\mcal{N})} &&& {\Alg^\aug(\enEnd^\mrm{L}_\calA(\mcal{N}))}
	\arrow["{\Sigma^{-1}}"', from=1-1, to=2-1]
	\arrow["{\Fixcolax_{(-)}(\mcal{N})}", from=1-1, to=1-4]
	\arrow["{\Lambda(-)}"', from=2-1, to=2-4]
	\arrow["{\mnd(-)}", from=1-4, to=2-4]
    \end{tikzcd}\]
\end{prop}
\begin{proof}   
    The top horizontal and right vertical functors are provided by \cref{prop:fixcolax-stable-setting} and \cref{const:monad-from-adjunction} respectively.
    Tracing through the definitions we see that it suffices to show that the composite functor
    \[ 
    \enEnd^\mrm{L}_\calA(\mcal{N}) \xrightarrow{\Fixcolax_{(-)}(\mcal{N})}  \left(\enRMod_{\calA}(\PrLst)^{\ladj}_{\mcal{N}/}\right)_{/\Id_\mcal{N}} \xrightarrow{\mnd(-)} \Alg^\aug(\enEnd^\mrm{L}_\calA(\mcal{N})) \xrightarrow{\overline{(-)}} \enEnd^\mrm{L}_\calA(\mcal{N})\]
    is equivalent to $\Sigma^{-1} \colon \enEnd^\mrm{L}_\calA(\mcal{N}) \to \enEnd^\mrm{L}_\calA(\mcal{N})$.
    To see this, we precompose the cofiber sequence of \cref{prop:cofib-seq-laxfix} with $i_!$ to deduce
    \[ 
    \overline{\mnd}(i_! \colon \mcal{N} \adj \Fixcolax_T(\mcal{N}) \colon i^\ast)= \fib(i^\ast i_! \to \Id_\mcal{N}) \simeq \Sigma^{-1}T.\qedhere
    \]
\end{proof}

\begin{lem}\label{lem:action-on-functors}
	Let $\mcal{U} \in \RMod_{\mcal{A}}(\PrLst)$, 
	let $\mcal{V} \in \BMod{\mcal{E}}{\mcal{A}}(\PrLst)$,
	and let $E \in \Alg(\mcal{E})$. 
	There exists a canonical equivalence 
	\[	\Fun^\mrm{L}_\calA(\mcal{U},\LMod_E(\mcal{V})) \simeq \LMod_E(\Fun^\mrm{L}_\calA(\mcal{U},\mcal{V})) \]
\end{lem}
\begin{proof}    
    By \cite[Theorem 4.8.4.1]{HA} we have
	\begin{align*}
		\Fun_{\RMod_{\mcal{A}}(\PrLst)}(\mcal{U},\LMod_E(\mcal{V})) 
		& \simeq \Fun_{\RMod_{\mcal{A}}(\PrLst)}(\mcal{U},\Fun_{\LMod_{\mcal{E}}(\PrLst)}(\RMod_E(\mcal{E}),\mcal{V})) \\
		&\simeq \Fun_{\BMod{\mcal{E}}{\mcal{A}}(\PrLst)}(\mcal{U} \otimes \RMod_E(\mcal{E}),\mcal{V})\\
		& \simeq \Fun_{\LMod_{\mcal{E}}(\PrLst)}(\RMod_E(\mcal{E}),\Fun_{\RMod_{\mcal{A}}(\PrLst)}(\mcal{U},\mcal{V}))\\
		& \simeq \LMod_E(\Fun_{\RMod_{\mcal{A}}(\PrLst)}(\mcal{U}.\mcal{V})).\qedhere
	\end{align*}
\end{proof}

\begin{defn}\label{defn:internally-coreflective}
    An adjunction $f_! \colon X \adj Y \colon f^\ast$ in an \twocategory{} $\frX$ is called \textit{\hl{reflective}} 
    if the counit map $c_f \colon f_!f^\ast \to \id_Y$ is an equivalence. 
    It is called \textit{\hl{coreflective}} if the unit $\id \colon \id_X \to f^\ast f_!$ is an equivalence. 
    We will often use the term \textit{\hl{internally (co-)reflective adjunction}} to emphasize that the adjunction is happening in $\frX$ rather than in some larger ambient \twocategory{}.
\end{defn}

\begin{lem}\label{lem:recollement-from-adjunction}
    Let $i_! \colon \calR \adj \calC \colon i^\ast$ be an internally coreflective adjunction in $\enRMod_\calA(\PrLst)$. 
    Then:
    \begin{enumerate}
        \item 
        The full subcategory $\calL \coloneqq  \ker(i^\ast) \subseteq \calC$ is a presentable right $\calA$-module and the inclusion extends to an internally reflective adjunction in $\enRMod_\calA(\PrLst)$:
        \[\begin{tikzcd}
    	{ j^! \colon \calC} && {\calL \colon j_!}
    	\arrow[""{name=0, anchor=center, inner sep=0}, shift left=2, from=1-1, to=1-3]
    	\arrow[""{name=1, anchor=center, inner sep=0}, shift left=2, hook', from=1-3, to=1-1]
    	\arrow["\dashv"{anchor=center, rotate=-90}, draw=none, from=0, to=1]
        \end{tikzcd}\]
        \item 
        We have a canonical recollement of stable \categories{}:
        \[\begin{tikzcd}
    	\calL && \calC && \calR
    	\arrow["{i^\ast}"{description}, from=1-3, to=1-5]
    	\arrow["{i_\ast}"{description}, shift right=4, hook', from=1-5, to=1-3]
    	\arrow["{j_!}"{description}, hook, from=1-1, to=1-3]
    	\arrow["{j^\ast}"{description}, shift right=4, from=1-3, to=1-1]
    	\arrow["{j^!}"{description}, shift left=4, from=1-3, to=1-1]
    	\arrow["{i_!}"{description}, shift left=4, hook', from=1-5, to=1-3]
        \end{tikzcd}\]
        \item 
        The following square in $\RMod_\calA(\PrLst)$
        \[\begin{tikzcd}
    	\calR & \calC \\
    	0 & \calL
    	\arrow["{i_!}", hook, from=1-1, to=1-2]
    	\arrow["{j^!}", from=1-2, to=2-2]
    	\arrow[from=1-1, to=2-1]
    	\arrow[from=2-1, to=2-2]
        \end{tikzcd}\]
        is both a pushout and a pullback.
    \end{enumerate}
\end{lem}
\begin{proof}
    $(1)$
    The kernel of the internal right adjoint $i^\ast$ is closed under both tensors and cotensors, hence the inclusion
    $j_! \colon \calL=\ker(i^\ast) \to \calC$ 
    preserves both tensors and cotensors.
    It follows that both $j_!$ and $j^!$ are $\calA$-linear. The rest follows from $(2)$.
    
    $(2)$
    Since $i^\ast$ preserves limits and colimits it follows from \cite[Theorem 1.1]{reflect} that the inclusion $j_! \colon \calL \coloneqq  \ker(i^\ast) \hookrightarrow \calC$ admits adjoints on both sides
    $j^\ast \dashv j_! \dashv j^!$.
    By \cite[Proposition A.8.20.]{HA} it follows that the inclusions 
    $\calL=\im(j_!) \hookrightarrow \calC$ and 
    $\calR=\im(i_\ast) \hookrightarrow \calC$ 
    exhibit $\calC$ as a recollement.
    $(3)$ 
    Recall that limits in $\PrLst$ are computed levelwise and colimits in $\PrLst$ are computed as limits of the right adjoints.
    It follows that the forgetful $\RMod_\calA(\PrLst) \to \PrLst$ creates and preserves both limits and colimits.
    The claim now follow by combining $(1)$ and $(2)$.
\end{proof}

\begin{prop}\label{prop:enriched-barr-beck}
	Let $f^{\ast} : \mcal{U} \to \mcal{V}$ be an internal right adjoint in  $\enRMod_{\mcal{A}}(\PrLst)$.
	Then $f^{\ast}$ admits a canonical lift to an internally coreflective adjunction in $\enRMod_\calA(\PrLst)$:
	\[\begin{tikzcd}
	{\tild{f}_! \colon \LMod_{\mnd(f_!\dashv f^\ast)}(\mcal{V})} && {\calU \colon \tild{f}^\ast}
	\arrow[""{name=0, anchor=center, inner sep=0}, shift left=2, hook, from=1-1, to=1-3]
	\arrow[""{name=1, anchor=center, inner sep=0}, shift left=2, from=1-3, to=1-1]
	\arrow["\dashv"{anchor=center, rotate=-90}, draw=none, from=0, to=1]
    \end{tikzcd}\]
\end{prop}
\begin{proof}
    Applying \cref{lem:action-on-functors} to the canonical 
    action of the algebra $\mnd(f_!\dashv f^\ast) \in \Alg(\enEnd^\mrm{L}_\calA(\mcal{V}))$ on the right adjoint $f^{\ast} \in \Fun^\mrm{L}_\calA(\calU,\calV)$ yields an $\mcal{A}$-linear factorization:
    \[f^\ast \colon \calU \xrightarrow{\tild{f}^\ast}\LMod_{\mnd(f_!\dashv f^\ast)}(\calV) \xrightarrow{\forget} \calV \]
	The composite $f^\ast$ preserves colimits, hence by \cite[Theorem 4.7.3.5]{HA}, $\tild{f}^\ast$ admits a (necessarily $\calA$-oplax-linear) fully faithful left adjoint which we denote by $\tild{f}_! \colon \LMod_{\mnd(f_!\dashv f^\ast)}(\calV) \hookrightarrow \calU$.
	We must show that $\tild{f}_!$ is in fact $\calA$-linear.
	Since $f^\ast$ is an internal right adjoint it preserves cotensors.
	The forgetful functor $\LMod_{\mnd(f_!\dashv f^\ast)}(\mcal{V}) \to \mcal{V}$ is a conservative internal right adjoint, hence creates and preserves cotensors. 
	It follows that $\tild{f}^\ast$ also preserve cotensors, and thus, its left adjoint $\tild{f}_!$ is $\calA$-linear.
\end{proof}

\begin{thm}\label{cor:bifiber-seq-general-case}
    The adjunction 
    $i_! \colon \mcal{N} \adj \Fixcolax_T(\mcal{N}) \colon i^\ast$
    from \cref{const:defn-of-section-for-colaxfix} lifts canonically to an internally coreflective adjunction in $\enRMod_\calA(\PrLst)$:
    \[\begin{tikzcd}
	{\tild{i}_!\colon \LMod_{\Lambda(\Sigma^{-1}T)}\left(\mcal{N} \right)} && { \Fixcolax_{T}\left(\mcal{N}\right) \colon \tild{i}^\ast}
	\arrow[""{name=0, anchor=center, inner sep=0}, shift left=2, hook, from=1-1, to=1-3]
	\arrow[""{name=1, anchor=center, inner sep=0}, shift left=2, from=1-3, to=1-1]
	\arrow["\dashv"{anchor=center, rotate=-90}, draw=none, from=0, to=1]
    \end{tikzcd}\]
    The essential image of $\tild{i}_!$ consists of pairs $(X,\varphi) \in \Fixcolax_T(\mcal{N})$ satisfying:
    \[X[\varphi^{-1}] = \colim\left(X \xrightarrow{\varphi} T(X) \xrightarrow{T(\varphi)} T^2(X) \to \cdots \right)\simeq 0\]
\end{thm}
\begin{proof}
    By \cref{thm:fixcolax-categorifies-exterior},
    $i^\ast \colon \Fixcolax(\mcal{N}) \to \mcal{N}$ 
    is internally right adjoint and
    $\mnd(i_! \dashv i^\ast) \simeq \Lambda(\Sigma^{-1}T) \in \Alg(\End^\mrm{L}_\calA(\mcal{N}))$.
    \cref{prop:enriched-barr-beck} then provides an internally coreflective adjunction in $\enRMod_\calA(\PrLst)$:
    \[\begin{tikzcd}
	{\tild{i}_! \colon \LMod_{\Lambda(\Sigma^{-1}T)}(\mcal{N})} && {\Fixcolax_T(\mcal{N}) \colon \tild{i}^\ast}
	\arrow[""{name=0, anchor=center, inner sep=0}, shift left=2, hook, from=1-1, to=1-3]
	\arrow[""{name=1, anchor=center, inner sep=0}, shift left=2, from=1-3, to=1-1]
	\arrow["\dashv"{anchor=center, rotate=-90}, draw=none, from=0, to=1]
    \end{tikzcd}\]
    It remains to identify the essential image. 
    By \cref{prop:fixcolax-stable-setting} we have 
    $(X,\varphi) \in \ker(i^\ast)$ if and only if $\varphi \colon X \to T(X)$ is an equivalence, i.e. $\ker(i^\ast) = \Fix_T(\mcal{N})$.
    We saw in \cref{obs:shift-functors} that the reflection onto $\Fix_T(\mcal{N})$ is given by $(X,\varphi) \mapsto \sigma^\infty(X,\varphi)$ hence by \cref{lem:recollement-from-adjunction} we have $\im(\tild{i}_!) = \ker(\sigma^\infty)$.
    It follows that $(X,\varphi)$ lies in the essential image if and only if $X[\varphi^{-1}] \simeq 0$, as promised.
\end{proof}

\begin{rem}
    \cref{cor:bifiber-seq-general-case} can be formulated in terms of recollements.
    Namely, the fully faithful functors:
    \[ \mrm{include} \colon \Fix_T(\mcal{N}) \hookrightarrow \Fixcolax_T(\mcal{N}) \qquad \text{ and } \qquad \tild{i}_\ast \colon \LMod_{\Lambda(\Sigma^{-1}T)}(\mcal{N}) \to \Fixcolax_T(\mcal{N})\] 
    exhibit $\Fixcolax_T(\mcal{N})$ as a stable recollement in the sense of \cite[Definition A.8.1]{HA}.
    \cite[Proposition A.8.11]{HA} then provides a pullback square:
    \[\begin{tikzcd}
	{\Fixcolax_T(\mcal{N})} && {\Fix_T(\mcal{N})^{\Delta^1}} \\
	{\LMod_{\Lambda(\Sigma^{-1}T)}(\mcal{N})} && {\Fix_T(\mcal{N})}
	\arrow[from=1-1, to=1-3]
	\arrow["{\ev_1}", from=1-3, to=2-3]
	\arrow["\sigma^{\infty}\circ\tild{i}_{\ast}", from=2-1, to=2-3]
	\arrow[from=1-1, to=2-1]
	\arrow["\lrcorner"{anchor=center, pos=0.125}, draw=none, from=1-1, to=0]
    \end{tikzcd}\]
\end{rem}

\begin{lem}\label{lem:identify-obstruction-map}
    The composite functor 
    $\LMod_{\Lambda(\Sigma^{-1}T)}(\mcal{N}) \xrightarrow{\tild{i}_!} \Fixcolax_{T}(\mcal{N}) \xrightarrow{\pi_!} \Ar(\mcal{N})$
    is represented by the edge map $\Id_\calN \to T$ of the canonical cofiber sequence in $\RMod_{\Lambda(\Sigma^{-1}T)}( \enEnd^\mrm{L}_\calA(\mcal{N}))$:
    \[\Sigma^{-1} T \too\Lambda(\Sigma^{-1}T) \too \Id_\calN\]
\end{lem}
\begin{proof}
    The canonical equivalence
    \[\Fun^\mrm{L}_\calA\left(\Fixcolax_T(\mcal{N}) ,\Ar\left(\mcal{N}\right)\right)\simeq \Ar\left(\Fun^\mrm{L}_\calA(\Fixcolax_T(\mcal{N}) ,\mcal{N})\right)\] 
	identifies the functor 
	$\pi_! \colon  \Fixcolax_T(\mcal{N}) \to \Ar\left(\mcal{N}\right)$ 
	with 
	$\gamma_\can : q_! \to T q_!$.	
	By \cref{prop:cofib-seq-laxfix}, we have a cofiber sequence 
	$i^{\ast} \xrightarrow{\delta} q_! \xrightarrow{\gamma_\can} T q_!$ 
	where $\delta : i^{\ast} \xrightarrow{i^{\ast}(u_q)} i^{\ast}q^{\ast}q_! \simeq q_!$. 
	Meanwhile, the augmentation 
	$\epsilon \colon \Lambda(\Sigma^{-1}T) \to \Id_\calN$
	considered as a natural transformation in $\Ar\left(\Fun^\mrm{L}_\calA(\LMod_{\Lambda(\Sigma^{-1}T)}(\mcal{N}) ,\mcal{N})\right)$
	corresponds to 
	$\lambda \colon \nu^\ast \xrightarrow{\nu^{\ast}(u_\epsilon)} \nu^{\ast} \epsilon^{\ast} \epsilon_! \simeq \epsilon_!$ 
	where we write $\nu \colon \Id_\calN \to\Lambda(\Sigma^{-1}T)$ for the unit of the algebra.
	We must therefore show that 
	$\lambda \colon \nu^\ast \to \epsilon_!$ 
	is equivalent to
	$\delta (\tild{i}_!) \colon  q_! \tild{i}_! \to T q_! \tild{i}_!$
    which can be seen from the following diagram:
	\[\begin{tikzcd}
		& {i^{\ast}\tild{i}_!} & {i^{\ast}\underbracket{q^{\ast}q_!}\tild{i}_!} && {q_! \tild{i}_!} \\
		{\nu^{\ast}} & {\nu^{\ast}\underbracket{\tild{i}^{\ast} \tild{i}_!}} & {\nu^{\ast}\tild{i}^{\ast}\underbracket{q^{\ast}q_!}\tild{i}_!} & {\nu^{\ast}\epsilon^{\ast}\epsilon_!} & {\epsilon_!}
		\arrow["\simeq", from=2-1, to=2-2]
		\arrow["\simeq"', from=1-2, to=2-2]
		\arrow["\simeq", from=2-2, to=2-3]
		\arrow[from=1-2, to=1-3]
		\arrow["\simeq", from=1-3, to=1-5]
		\arrow["\simeq"', from=1-5, to=2-5]
		\arrow["\simeq", from=2-4, to=2-5]
		\arrow["\simeq", from=2-3, to=2-4]
		\arrow["\simeq"', from=1-3, to=2-3]
		\arrow["{\nu^{\ast}(u_\epsilon)}"', curve={height=24pt}, from=2-1, to=2-4]
		\arrow["\delta(\tild{i}_!)", curve={height=-18pt}, from=1-2, to=1-5]
	\end{tikzcd}\]
\end{proof}

\section{The general case}\label{section:3}

The main goal of this section is to state and prove \cref{thm:sqz-general-case} - the generalization of \cref{thm:spectra-sqz-modules} promised in the introduction.
The zoo of subtle null homotopies appearing in the formulation and proof of \cref{thm:sqz-general-case} requires delicate care.
To alleviate some of the confusion we choose to begin this section with a systematic analysis of certain abstract generalization of \cref{defn:null-category-sqz}. 
We then specialize to the setting of modules over square zero extensions where we use the results of \cref{section:2} to prove \cref{thm:sqz-general-case}.
Finally, we explain how to deduce \cref{thm:sqz-connective-variant} using simple connectivity arguments.

\subsection{Generalities on null homotopies}

\begin{defn}\label{defn:null-category}
    A \textit{\hl{null datum}} is a triple $(\calD,E,\alpha)$ consisting of stable presentable \category{} 
	$\calD$, a left adjoint endomorphism
	$E \colon \calD \to \calD$ 
	and a natural transformation
	$\alpha \colon \Id \to E$.
	Given a null datum $(\calD,E,\alpha)$
	we define
	$\Null_{\alpha}(\mcal{D})$ 
	as the pullback:
		\[\begin{tikzcd}
			{\hl{\Null_{\alpha}(\mcal{D})}} & {\calD^{\partial \Delta^1}} \\
			{\calD} & {\calD^{\Delta^1}}
			\arrow["(-) \xrightarrow{0} (-)", from=1-2, to=2-2]
			\arrow[from=1-1, to=2-1]
			\arrow[from=1-1, to=1-2]
			\arrow[""{name=0, anchor=center, inner sep=0}, "{\Id \xrightarrow{\alpha} E}", from=2-1, to=2-2]
			\arrow["\lrcorner"{anchor=center, pos=0.125}, draw=none, from=1-1, to=0]
		\end{tikzcd}\]
	We write objects of $\Null_{\alpha}(\calD)$ as pairs $(X,h)$ where $X \in \calD$ and $h: \alpha_X \simeq 0 \in \Map_{\calD}(X,T(X))$.
\end{defn}

\begin{example}
    In the introduction we showed how to construct from a square zero datum $(A,I,\eta)$ in $\Sp$ a null datum $(\LMod_A,\Sigma^2 I \otimes (-), \theta_\eta)$.
    Substituting this null datum in \cref{defn:null-category}
    recovers
    \cref{defn:null-category-sqz} from the introduction. 
\end{example}

The forgetful functor $U \colon \Null_\alpha(\calD) \to \calD$ is conservative, hence for any $X \in \calD$, the fiber $\hl{\Null(\alpha_X)} \coloneqq  \{X\} \times_{\calD} \Null_\alpha(\calD)$ is an $\infty$-groupoid.
Specifically, $\Null(\alpha_X)$ is the space of null homotopies of the map $\alpha_X \colon X \to E(X)$.
Just as we explained in the introduction, $\Null(\alpha_X)$ is naturally a torsor under $\Map_\calD(X,\Sigma^{-1}E(X))$ and thus given $h_1, h_2 \in \Null(\alpha_X)$ 
we have a well defined difference
$h_2-h_1 \coloneqq  h_1^{-1}\circ h_2 \in \Map_\calD(X, \Sigma^{-1} E(X))$.

\begin{obs}\label{const:functor-to-null}
    A morphism
    $\calC \too \Null_\alpha(\calD)$ in $\xPrL{}{\st}$
    consists of a pair $(f_!,\gamma)$ where $\hl{f_!} \colon  \calC \too \calD$ is a left adjoint functor and $\hl{\gamma} \colon 0 \simeq \alpha(f_!)$ is a null homotopy of $\alpha(f_!)\colon f_! \too E f_!$.
    Conversely, every such a pair defines a unique left adjoint functor denoted
    $\hl{\Phi(f_!,\gamma)} \colon \calC \too \Null_\alpha(\calD)$.
\end{obs}

\begin{example}
    Let $(\calD,E,\alpha)$ be a null datum.
    The forgetful functor 
    $U \colon \Null_\alpha(\calD) \to \calD$ 
    comes with a canonical null homotopy 
    $\hl{h_\can} \colon 0 \simeq \alpha(U) \colon U \to EU$ 
    characterized by the property that $(U,h_\can)$ classifies the identity functor: 
    \[\Phi(U,h_\can) \simeq \Id_{\Null_\alpha(\calD)} \colon \Null_\alpha(\calD) \too \Null_\alpha(\calD)\]
\end{example}

\begin{prop}\label{lem:weird-adjoint-lemma}
    Let $(\calD,E,\alpha)$ be a null datum and let $(f_!,\gamma)$ be a pair, 
    as in \cref{const:functor-to-null}, 
    defining a left adjoint functor  $\Phi(f_!,\gamma) \colon \calC \to \Null_\alpha(\calD) $.
    The right adjoint 
    $\Phi(f_!,\gamma)^R \colon \Null_\alpha(\calD) \too \calC$ 
    is given by
    \[\Phi(f_!,\gamma)^R \simeq \fib\left(\lambda_{(f_!,\gamma)}(U) - f^\ast h_\can \colon f^\ast U \to \Sigma^{-1} f^\ast E U \right)\]
    where 
    $\lambda_{(f_!,\gamma)} \colon 0 \simeq f^\ast \alpha \colon f^\ast \to f^\ast E $ 
    is the null homotopy defined by the following diagram:
    \[\begin{tikzcd}
	{f^\ast} && {f^\ast f_! f^\ast} \\
	\\
	{f^\ast E} && {f^\ast E f_! f^\ast}
	\arrow[""{name=0, anchor=center, inner sep=0}, "{f^\ast \alpha(f_! f^\ast)}"', from=1-3, to=3-3]
	\arrow["{f^\ast \alpha}"', from=1-1, to=3-1]
	\arrow["{u_f(f^\ast)}", from=1-1, to=1-3]
	\arrow["{f^\ast E c_f}", from=3-3, to=3-1]
	\arrow[""{name=1, anchor=center, inner sep=0}, "0", curve={height=-30pt}, from=1-3, to=3-3]
	\arrow["f^\ast\gamma"', shorten <=6pt, shorten >=6pt, Rightarrow, from=1, to=0]
    \end{tikzcd}\]
\end{prop}
\begin{proof}
    Consider the functor $f^\ast j_\alpha^\ast j_{\alpha,!} \colon \calD \xrightarrow{j_{\alpha,!}} \calD^{\Delta^1} \xrightarrow{j_\alpha^\ast} \calD \xrightarrow{f^\ast} \calC$.
    By the diagram at the end of
    \cref{lem:right-adjoint-from-arrow-category}, we have a pullback square: 
    \[\begin{tikzcd}
	{f^\ast j_\alpha^\ast j_{\alpha,!}} & {f^\ast E^R E} \\
	{f^\ast} & {f^\ast E}
	\arrow["{f^\ast(\alpha^R(E))}", from=1-2, to=2-2]
	\arrow[from=1-1, to=2-1]
	\arrow[from=1-1, to=1-2]
	\arrow["{f^\ast(\alpha)}"', from=2-1, to=2-2]
    \end{tikzcd}\]
    We claim the right vertical natural transformations is canonically null homotopic.
    Indeed $\gamma$ provides such a null homotopy via the following commutative diagram:
    \[\begin{tikzcd}
	{f^\ast E^R E} &&&& {f^\ast f_!f^\ast E^R E} \\
	\\
	{f^\ast E} && {f^\ast E E^R E} && {f^\ast E f_! f^\ast E^R E}
	\arrow[""{name=0, anchor=center, inner sep=0}, "{f^\ast \alpha(f_! f^\ast E^R E )}"', from=1-5, to=3-5]
	\arrow["{f^\ast E c_f(E^R E)}", from=3-5, to=3-3]
	\arrow["{u_f(f^\ast T^R E)}", from=1-1, to=1-5]
	\arrow["{f^\ast c_E(E)}", from=3-3, to=3-1]
	\arrow[""{name=1, anchor=center, inner sep=0}, "0", curve={height=-30pt}, from=1-5, to=3-5]
	\arrow["{f^\ast \alpha^R(E)}", from=1-1, to=3-1]
	\arrow["{\gamma}"', shorten <=6pt, shorten >=6pt, Rightarrow, from=1, to=0]
    \end{tikzcd}\]
    This gives a canonical equivalence 
    $f^\ast j_\alpha^\ast j_{\alpha,!} \simeq \fib(f^\ast \alpha \colon f^\ast \to f^\ast E) \oplus f^\ast E^R E$.
    By adjoint descent the right adjoint 
    $\Phi(f_!,\gamma)^R$ 
    sits in a pullback square:
    \[\begin{tikzcd}
	{\Phi(f_!,\gamma)^R} && {f^\ast U\oplus f^\ast E^R EU} \\
	{f^\ast U} & {f^\ast j^\ast_{\alpha} j_{\alpha,!}U} & {\fib\left(f^\ast \alpha(U) \colon f^\ast U \to f^\ast EU\right) \oplus f^\ast E^R EU}
	\arrow[from=1-1, to=1-3]
	\arrow["{f^\ast u_{j_\alpha} (U)}", from=2-1, to=2-2]
	\arrow[from=1-1, to=2-1]
	\arrow["\simeq", from=2-2, to=2-3]
	\arrow[from=1-3, to=2-3]
    \end{tikzcd}\]
    The right vertical natural map is induced by $f^\ast$ applied to the unit of the adjunction 
    $\calD^{\partial \Delta^1} \to \calD^{\Delta^1}$ 
    whose left and right adjoints are given by 
    $(X,Y) \longmapsto \left(X \xrightarrow{0} Y \right)$ 
    and
    $\left(X \xrightarrow{\gamma_\can} Y\right) \longmapsto \left(\fib(\gamma_\can),Y\right)$ 
    respectively.
    Subtracting $f^\ast E^R E U$ from the right hand side we obtain the following pullback square:
    \[\begin{tikzcd}
	{\Phi(f_!,\gamma)^R} & {f^\ast U } \\
	{f^\ast U} & {\fib\left(f^\ast \alpha(U) \colon f^\ast U \to f^\ast EU\right)}
	\arrow[from=1-2, to=2-2]
	\arrow[from=2-1, to=2-2]
	\arrow[from=1-1, to=1-2]
	\arrow[from=1-1, to=2-1]
    \end{tikzcd}\]
    where the bottom horizontal and right vertical maps correspond to null homotopies of $f^\ast \alpha(U)$. 
    The null homotopy corresponding to the bottom horizontal map is induced by $\gamma$ via the following diagram:
    
    We denote this null homotopy by $\lambda_{(f_!,\gamma)} \colon 0 \simeq f^\ast \alpha \colon f^\ast \to f^\ast E $.
    Meanwhile the null homotopy corresponding to the right vertical map is given by 
    $f^\ast h_\can \colon 0 \simeq f^\ast \alpha(U) \colon f^\ast U \to f^\ast EU$.
    Subtracting them gives the desired pullback square:
     \[\begin{tikzcd}
	{\Phi(f_!,\gamma)^R} &&& {0} \\
	{f^\ast U} &&& {\Sigma^{-1} f^\ast E U}
	\arrow[from=1-1, to=1-4]
	\arrow["{\lambda_{(f_!,\gamma)}(U)-f^\ast (h_\can)}", from=2-1, to=2-4]
	\arrow[from=1-1, to=2-1]
	\arrow[from=1-4, to=2-4]
    \end{tikzcd}\]
\end{proof}

\begin{cor}\label{cor:fiber-sequence-right-adjoint}
    Let $(\calD,E,\alpha)$ be a null datum and let $(f_!,\gamma)$ be a pair, 
    as in \cref{const:functor-to-null}, 
    defining a left adjoint functor  $\Phi(f_!,\gamma) \colon \calC \to \Null_\alpha(\calD) $.
    There is a canonical fiber sequence:
    \[\Phi(f_!,\gamma)^R \too  f^\ast U \too \Sigma^{-1} f^\ast EU\]
\end{cor}

\subsection{Modules over square zero extensions}
In this section we prove \cref{thm:sqz-general-case} - the main result of the paper.
We will henceforth fix a stable presentably monoidal \category{} $\calA$.
We begin this section by revisiting the constructions from the introduction in this more general setting.

Recall that given a square zero datum $(A,I,\eta)$ in $\calA$ we defined the associated square zero extension as the pullback in $\Alg(\calA)_{/A}$:
\begin{equation}\label{diag:defn-of-A-eta}
        \begin{tikzcd}
    	{A^\eta} & A \\
    	A & {A \ltimes \Sigma I}
    	\arrow["\eta"', from=2-1, to=2-2]
    	\arrow["p"', from=1-1, to=2-1]
    	\arrow["p", from=1-1, to=1-2]
    	\arrow["{\eta_0}", from=1-2, to=2-2]
    	\arrow["\lrcorner"{anchor=center, pos=0.125}, draw=none, from=1-1, to=2-2]
        \end{tikzcd}\tag{$\star$}
    \end{equation}
The forgetful functors $\Alg(\calA)_{/A} \to \Alg(\calA)$ and $\Alg(\calA)_{A^\eta/} \to \Alg(\calA)$ create and preserve weakly contractible limits, hence we may alternatively regard this square as a pullback in $\Alg(\calA)_{A^\eta/}$.
In particular we may view it as a pullback square of $(A^\eta,A^\eta)$-bimodules. 

\begin{obs}\label{obs:square-bimodules-Aeta}
    The square \hyperref[diag:defn-of-A-eta]{($\star$)} is sent by the forgetful functor
    $\Alg(\calA)_{A^\eta/} \to \BMod{A^\eta}{A^\eta}(\calA)$ to the following pullback square:
    \[\begin{tikzcd}
	{\hl{A^\eta}} & (p,p)^\ast A\\
	(p,p)^\ast A & {(\widetilde{p},\widetilde{p})^\ast A \ltimes \Sigma I}
	\arrow[from=2-1, to=2-2]
	\arrow[from=1-1, to=2-1]
	\arrow[from=1-1, to=1-2]
	\arrow[from=1-2, to=2-2]
	\arrow["\lrcorner"{anchor=center, pos=0.125}, draw=none, from=1-1, to=2-2]
    \end{tikzcd}\]
    where $\hl{\widetilde{p}}\coloneqq  \eta_0 \circ p \simeq \eta \circ p \colon A^\eta \to A \ltimes \Sigma I$.
\end{obs}

\begin{notation}\label{notation:bimodule-precise}
    Given $R,S \in \Alg(\mcal{A})$
    we have by \cite[Theorem 4.3.2.6 \& 4.3.2.7]{HA} a canonical equivalence: \[\LMod_R(\calA) \otimes_\calA \RMod_S(\calA) \simeq \BMod{R}{S}(\calA)\]
    Consequently, the composite functor
    \[\Alg(\calA) \times \Alg(\calA) \xrightarrow{\left(\LMod_{(-)}(\mcal{A}), \RMod_{(-)}(\mcal{A})\right)} \RMod_\calA(\PrLst) \times \LMod_\calA(\PrLst) \xrightarrow{(-)\otimes_\calA(-)} \PrLst\]
    provides a two-variable functoriality for bimodules:
    \[\hl{\BMod{(-)}{(-)}(\calA)} \colon \Alg(\calA) \times \Alg(\calA) \too \PrLst, \qquad (R,S) \longmapsto \BMod{R}{S}(\calA)\]
    Given a pair of morphism $f \colon R \to S$ and $f' \colon R' \to S'$ we denote the corresponding adjunction by:
    \[ \hl{(f,f')_!} \colon \BMod{R}{S}(\calA) \adj \BMod{R'}{S'}(\calA) \colon \hl{(f,f')^\ast} \]
\end{notation}

\begin{defn}
    Given $A \in \Alg(\calA)$ and $M \in \BMod{A}{A}(\calA)$ we denote by $\hl{\theta^M} \colon A \to \Sigma M$ the edge map of the canonical fiber sequence in $\BMod{A \ltimes M}{A \ltimes M}(\calA)$:
    \[J \too A \ltimes M \too A\]
    We define the \hl{\textit{(left) obstruction map}} of a square zero datum $(A,I,\eta)$ to be the $(A,A)$-bimodule map:
    \[
    \hl{\theta_\eta}\coloneqq (\eta_0,\eta)^\ast(\theta^{\Sigma I}) \colon A \too \Sigma^2 I \in \BMod{A}{A}(\calA) 
    \]
\end{defn}

\begin{const}\label{ex:weird-adjoint-sqz}
    Let $A \in \Alg(\calA)$ and let $I \in \BMod{A}{A}(\calA)$.
    The image of the canonical fiber sequence
    \[ \Sigma I \too A \ltimes (\Sigma I) \too A\quad \in \BMod{A \ltimes \Sigma I}{A \ltimes \Sigma I}\]
    under the forgetful functor $\BMod{A \ltimes \Sigma I}{A \ltimes \Sigma I}(\calA) \too \calA$ admits a canonical splitting.
    Consequently, if $(A,I,\eta)$ is a square zero datum, the image of the associated obstruction map $\theta_\eta \colon A \to \Sigma^2 I \in \BMod{A}{A}(\calA)$ under the forgetful functor $\BMod{A}{A}(\calA) \to \calA$ is canonically null homotopic map.
    We denote this null homotopy by $\hl{\gamma_A}$.
    The pair $(A \otimes (-),\gamma_A)$ 
    defines an $\calA$-linear left adjoint functor (see \cref{const:functor-to-null}):
    \[\Phi(A\otimes (-),\gamma_A) \colon \mcal{A} \too \Null_{\theta_\eta}(\LMod_A(\mcal{A}))\]
    More generally, for $\mcal{M} \in \LMod_\calA(\PrLst)$ this gives a left adjoint functor:
    \[\Phi(A\otimes (-),\gamma_A) \colon \mcal{M} \too \Null_{\theta_\eta}(\LMod_A(\mcal{M}))\]
\end{const}

\begin{defn}\label{defn:beta-definition}
    Let $(A,I,\eta)$ be a square zero datum in $\calA$ and let $\mcal{M} \in \LMod_\calA(\PrLst)$.
    Given $(X,h) \in \Null_{\theta_\eta}(\mcal{M})$
    we define
    \[ \hl{\beta_h} \coloneqq  \gamma_A (X) - \Und(h) \colon \Und(X) \too \Und(\Sigma I \otimes_A X) \quad \in \calM \]
    where $\hl{\Und} \colon \LMod_A(\mcal{M}) \to \mcal{M}$ denotes the forgetful and $\hl{\gamma_A(X)} \colon 0 \simeq \Und(\theta_\eta(X))$ denotes the path in $\Map_{\mcal{M}}(\Und(X),\Und(\Sigma^2 I \otimes_A X))$ defined by the following diagram:
    \[\begin{tikzcd}
	\Und(X) && {\Und(A) \otimes \Und(X)} \\
	\\
	{\Und(\Sigma^2 I\otimes_A X)} && {\Und(\Sigma^2 I) \otimes \Und(X)}
	\arrow[""{name=0, anchor=center, inner sep=0}, "{\Und(\theta_\eta) \otimes \Und(X)}"', shift right=1, from=1-3, to=3-3]
	\arrow["{\Und(\theta_\eta(X))}"', from=1-1, to=3-1]
	\arrow[from=1-1, to=1-3]
	\arrow[from=3-3, to=3-1]
	\arrow[""{name=1, anchor=center, inner sep=0}, "0", curve={height=-30pt}, from=1-3, to=3-3]
	\arrow[shorten <=6pt, shorten >=6pt, Rightarrow, from=1, to=0]
    \end{tikzcd}\]
    where we pick the null homotopy on the right to be $\Und(\gamma_A) \otimes \Und(X) \colon 0 \simeq \Und(\theta_\eta) \otimes \Und(X)$.
\end{defn}

\begin{cor}\label{cor:weird-adjoint-special-case}
    The right adjoint of the functor
    $\Phi(A\otimes (-),\gamma_A) \colon \mcal{M} \to \Null_{\theta_\eta}(\LMod_A(\mcal{M}))$
    from \cref{ex:weird-adjoint-sqz}
    is given on objects by:
    \[\Phi(A \otimes (-),\gamma_A)^R \colon (X,h) \longmapsto \fib\left( \beta_h \colon \Und(X)  \to \Und(\Sigma I \otimes_A X)\right)\]
\end{cor}
\begin{proof}
    \cref{lem:weird-adjoint-lemma} provides the following description of the right adjoint:
    \[\Phi(A\otimes (-),\gamma_A)^R(X,h) \simeq \fib \left( \lambda_{(A \otimes (-) ,\gamma_A)}(U)_{(X,h)} - \Und(h) \colon  \Und(X) \too \Und(\Sigma I \otimes_A X)\right) \in \mcal{M}\]
    It remains to observe that 
    $\lambda_{(A \otimes (-) ,\gamma_A)}(U)_{(X,h)} = \gamma_A(X)$
    as introduced in \cref{defn:beta-definition}.
\end{proof}

Some aspects of the proofs of \cref{thm:sqz-general-case} and \cref{thm:sqz-stable-nonsplit-base} are elucidated by the following, seemingly ad-hoc, notion.

\begin{defn}
    A conical diagram of \categories{} and $\calC_{(-)} \colon \calJ^{\triangleleft} \to \Cat_\infty$ is said to be an \hl{\textit{almost limit diagram}} if the comparison functor 
    $\calC_{-\infty} \to \lim_{\alpha \in \calJ} \calC_\alpha$
    is fully faithful. 
    A commutative square which defines an almost limit diagram is called \hl{\textit{almost cartesian}}.
\end{defn}

The main utility of almost limits comes from the following pleasantly surprising property.

\begin{lem}\label{prop:almost-limit-tensoring}
    Let $\calU_{(-)} \colon \calJ^{\triangleleft} \to \RMod_\calA(\PrLst)^\ladj$ 
    be an almost limit diagram such that $\calJ$ is a finite \category{}, 
    and let
    $\calV \in \LMod_\calA(\PrLst)$.
    Then
    $\calU_{(-)} \otimes_\calA \calV \colon \calJ^\triangleleft \to \PrLst$ 
    is an almost limit diagram.
\end{lem}
\begin{proof}
    Given $\alpha \in \calJ$ we write 
    $L_\alpha \colon \calU_{-\infty} \adj \calU_{\alpha} \colon R_\alpha$ for the corresponding adjunction and $u_\alpha \colon \Id \to R_\alpha L_\alpha$ for the unit.
    Accordingly, we write
    $L_{-\infty} \colon \calU_{-\infty} \adj \lim_{\alpha \in \calJ} \calU_{\alpha} \colon R_{-\infty}$ for the limit adjunction and denote its unit by 
    $u_{-\infty} \colon \Id \to R_{-\infty} L_{-\infty}$.
    The functor
    $(-)\otimes_\calA \Id_\calV \colon \enEnd_{\enRMod_\calA(\PrLst)}(\calU_{-\infty}) \to  \enEnd_{\PrLst}(\calU_{-\infty} \otimes_\calA \calV)$
    preserves finite limits hence by adjoint descent we have
    \[u_\infty \otimes_\calA \Id_\calV \simeq (\lim_{\alpha \in \calJ}u_\alpha )\otimes_\calA \Id_\calV \simeq \lim_{\alpha \in \calJ} \left(u_\alpha \otimes_\calA \Id_\calV \right) \quad \in \Ar\left(\enEnd_{\PrLst}(\calU_{-\infty} \otimes_\calA \calV)\right)\]
    Since $u_{-\infty}$ is an equivalence by assumption, so is the right most map.
    It remains to observe that by adjoint descent the right most map is precisely the unit of the adjunction in question:
    \[ \lim_{\alpha \in \calJ}\left( L_{\alpha} \otimes_\calA \Id_\calV\right) \colon \calU_{-\infty} \otimes_\calA \calV \adj \lim_{\alpha \in \calJ} \left(\calU_{\alpha} \otimes_\calA \calV\right)\colon \lim_{\alpha \in \calJ}\left( R_{\alpha} \otimes_\calA \Id_\calV\right) \]
\end{proof}

\begin{defn}
    Let $(A,I,\eta)$ be a square zero datum in $\calA$, let $\mcal{M}\in \LMod_\calA(\PrLst)$ and let $(X,h) \in \Null_{\theta_\eta}(\LMod_A(\mcal{M}))$.
    \begin{itemize}
    \item $(X,h)$ is called \hl{\textit{$\beta$-divisible}} if $\beta_h : \Und(X) \to \Und(\Sigma I\otimes_A X)$ is an equivalence.
    \item
    $(X,h)$ is called \hl{\textit{$\beta$-torsion}} if it admits no non-trivial maps into $\beta$-divisible objects. 
    Namely, for any morphism $\psi \colon (X,h) \to (Y,g)$ whose target $(Y,g) \in \Null_{\theta_\eta}(\LMod_A(\calM))$ is $\beta$-divisible we have $\psi \simeq 0$.
    \end{itemize}
\end{defn}

\begin{thm}\label{thm:sqz-general-case}
    Let $(A,I.\eta)$ be a square zero datum in $\calA$ and let $\mcal{M} \in \LMod_\calA(\PrLst)$.
    Then there exists a canonically commuting square in $\PrLst$
    \[\begin{tikzcd}
		{\LMod_{A^{\eta}}(\mcal{M})} &&& {\LMod_{A}(\mcal{M})^{\partial \Delta^1}} \\
		{\LMod_{A}(\mcal{M})} &&& {\LMod_{A}(\mcal{M})^{\Delta^1}}
		\arrow["{A \otimes_{A^{\eta}}(-)}"', from=1-1, to=2-1]
		\arrow["{\theta_{\eta} \otimes_A(-)}", from=2-1, to=2-4]
		\arrow["{(-)\overset{0}{\to}(-)}", from=1-4, to=2-4]
		\arrow["{(A \otimes_{A^{\eta}}(-), \Sigma^2 I \otimes_{A^{\eta}}(-) )}", from=1-1, to=1-4]
	\end{tikzcd}\]
	which induces an internally coreflective adjunction in $\PrLst$:
    \[\begin{tikzcd}
	{\Phi \colon \LMod_{A^\eta}(\mcal{M})} && { \Null_{\theta_\eta}\left(\LMod_{A}(\mcal{M})\right) \colon \Phi^R}
	\arrow[""{name=0, anchor=center, inner sep=0}, shift left=2, hook, from=1-1, to=1-3]
	\arrow[""{name=1, anchor=center, inner sep=0}, shift left=2, from=1-3, to=1-1]
	\arrow["\dashv"{anchor=center, rotate=-90}, draw=none, from=0, to=1]
    \end{tikzcd}\]
    An object $(X,h)$ lies in the essential image of $\Phi$ if and only if it is $\beta$-torsion. 
\end{thm}

Our proof of \cref{thm:sqz-general-case} will proceed by first proving an $\calA$-linear version of the theorem in the special case where $\mcal{M}=\calA$ and then "spending" that $\calA$-linearity to deduce the general result.
We begin with a preliminary lemma.

\begin{lem}\label{lem:almost-cartesian-module-categories}
    For any cartesian square of $\bbE_1$-algebras in $\calA$:
    \[\begin{tikzcd}
	R & {R'} \\
	S & {S',}
	\arrow["f"', from=1-1, to=2-1]
	\arrow["j", from=1-1, to=1-2]
	\arrow["{f'}", from=1-2, to=2-2]
	\arrow["\lrcorner"{anchor=center, pos=0.125}, draw=none, from=1-1, to=2-2]
	\arrow[from=2-1, to=2-2]
    \end{tikzcd}\]
    the comparison functor $\LMod_{R}(\calA) \hookrightarrow \LMod_{S}(\calA) \times_{\LMod_{S'}(\calA)} \LMod_{R'}(\calA)$ is fully faithful.
\end{lem}
\begin{proof}
    The forgetful functor $\Alg(\calA)_{R/} \to \BMod{R}{R}$ preserves limits and thus sends pullback squares of $\bbE_1$-algebras to pushout squares of bimodules.
    The square of bimodules corresponding to the above square can be identified under the equivalence 
    $ \BMod{R}{R}(\calA) \simeq \enEnd_\calA^\mrm{L}\left(\LMod_{R}(\mcal{A})\right)$ \cite[Remark 4.8.4.9]{HA} with the commutative square of natural transformations
    \[\begin{tikzcd}
	\id & {j^\ast j_!} \\
	{f^\ast f_!} & {(f'\circ j)^\ast (f'\circ j)_!.}
	\arrow[from=1-1, to=2-1]
	\arrow[from=1-1, to=1-2]
	\arrow[from=1-2, to=2-2]
	\arrow[from=2-1, to=2-2]
    \end{tikzcd}\]
    By adjoint descent, the comparison map $\id \to  f^\ast f_! \times_{(f'\circ j)^\ast (f'\circ j)_!} j^\ast j_!$ is precisely the unit of the adjunction $\LMod_R(\calA) \adj \LMod_{S}(\calA) \times_{\LMod_{S'}(\calA)} \LMod_{R'}(\calA)$ so we are done.
\end{proof}

\begin{thm}\label{thm:sqz-stable-nonsplit-base}
    Let $(A,I,\eta)$ be a square zero datum in $\calA$.
    There exists a canonically commuting square in $\RMod_\calA(\PrLst)$:
    \[\begin{tikzcd}
		{\LMod_{A^{\eta}}(\mcal{A})} &&& {\LMod_{A}(\mcal{A})^{\partial \Delta^1}} \\
		{\LMod_{A}(\mcal{A})} &&& {\LMod_{A}(\mcal{A})^{\Delta^1}}
		\arrow["{A \otimes_{A^{\eta}}(-)}"', from=1-1, to=2-1]
		\arrow["{\theta_{\eta} \otimes_A(-)}", from=2-1, to=2-4]
		\arrow["{(X,Y) \mapsto (X \xrightarrow{0}Y)}", from=1-4, to=2-4]
		\arrow["{(A \otimes_{A^{\eta}}(-), \Sigma^2 I \otimes_{A^{\eta}}(-) )}", from=1-1, to=1-4]
	\end{tikzcd}\]
	which gives rise to a internally coreflective adjunction in $\enRMod_\calA(\PrLst)$:
	\[\begin{tikzcd}
	{\Phi \colon \LMod_{A^\eta}(\mcal{A})} && { \Null_{\theta_\eta}\left(\LMod_{A}(\mcal{A})\right) \colon \Phi^R}
	\arrow[""{name=0, anchor=center, inner sep=0}, shift left=2, hook, from=1-1, to=1-3]
	\arrow[""{name=1, anchor=center, inner sep=0}, shift left=2, from=1-3, to=1-1]
	\arrow["\dashv"{anchor=center, rotate=-90}, draw=none, from=0, to=1]
    \end{tikzcd}\]
\end{thm}
\begin{proof}
    Note that almost cartesian squares are closed under pasting.
	Consider the following diagram of left adjoints
	\[\begin{tikzcd}
		{\LMod_{A^{\eta}}(\mcal{A})} & {\LMod_{A}(\mcal{A})} & {\LMod_{A}(\mcal{A})} \\
		{\LMod_{A}(\mcal{A})} & {\LMod_{A \ltimes \Sigma I}(\mcal{A})} & {\Fixcolax_{\Sigma^2 I \otimes_A(-)}(\LMod_A(\mcal{A}))}
		\arrow[""{name=0, anchor=center, inner sep=0}, "{\tild{i}_!}", hook, from=2-2, to=2-3]
		\arrow["{i_!}", from=1-3, to=2-3]
		\arrow["{=}", from=1-2, to=1-3]
		\arrow["{\eta_{0,!}}"', from=1-2, to=2-2]
		\arrow["{\eta_{!}}", from=2-1, to=2-2]
		\arrow["p_!"', from=1-1, to=2-1]
		\arrow["p_!", from=1-1, to=1-2]
	\end{tikzcd}\]
	The bottom right horizontal map is fully faithful by \cref{cor:bifiber-seq-general-case} and thus the right square is almost cartesian.
	The left square is almost cartesian by \cref{lem:almost-cartesian-module-categories}, hence by pasting so is the outer rectangle.
    Consider now the following commutative diagram:
	\[\begin{tikzcd}
		{\LMod_{A^{\eta}}(\mcal{A})} && {\LMod_{A}(\mcal{A})} && {\LMod_{A}(\mcal{A})^{\partial\Delta^1}} \\
		{\LMod_{A}(\mcal{A})} && {\Fixcolax_{\Sigma^2 I \otimes_A(-)}(\LMod_A(\mcal{A}))} && {\LMod_{A}(\mcal{A})^{\Delta^1}} \\
		&& {\LMod_{A}(\mcal{A})} && {\LMod_{A}(\mcal{A})^{\partial\Delta^1}}
		\arrow["{i_!}"', from=1-3, to=2-3]
		\arrow["{p_!}"', from=1-1, to=2-1]
		\arrow["{p_!}", from=1-1, to=1-3]
		\arrow["{\widetilde{i}_!\eta_!}", from=2-1, to=2-3]
		\arrow["{(\Id,\Sigma^2 I \otimes_A (-))}", from=1-3, to=1-5]
		\arrow["{\pi_!}", from=2-3, to=2-5]
		\arrow["{(-)\overset{0}{\to}(-)}", from=1-5, to=2-5]
		\arrow["{(\ev_0,\ev_1)}", from=2-5, to=3-5]
		\arrow["{q_!}"', from=2-3, to=3-3]
		\arrow["{(\Id,\Sigma^2 I \otimes_A (-))}", from=3-3, to=3-5]
	\end{tikzcd}\]
	The bottom right square is cartesian by \cref{lem: formula-for-colax-fixpoints}
	and the outer right rectangle is cartesian since the vertical composites are identities.
	Consequently, the top right square is cartesian and since we just showed that the top left square is almost cartesian it follows by pasting that the top outer rectangle is almost cartesian. 
	Finally, by \cref{lem:identify-obstruction-map} the composite functor $\pi_! \tild{i}_! \theta_!$ is represented by $\theta_\eta \colon A \to \Sigma^2 I $, hence the top outer rectangle is precisely the square we promised.
	It remains to prove that the right adjoint
	\[
	\Phi^R \colon \Null_{\theta_{\eta}}(\LMod_A(\mcal{A})) \too \LMod_{A^{\eta}}(\mcal{A})
	\]
	is in fact an internal right adjoint in $\enRMod_\calA(\PrLst)$.
	By \cref{cor:fiber-sequence-right-adjoint} we have a fiber sequence:
    \[\Phi^R(-) \too p^\ast U(-) \too \Sigma^{-1} p^\ast(\Sigma^2 I \otimes_A U(-)) \]
    Since 
    $p^\ast \colon \LMod_A(\calA) \to \LMod_{A^\eta}(\calA)$ 
    and 
    $U \colon \Null_{\theta_\eta}(\LMod_A(\calA)) \to \LMod_A(\calA)$ 
    are both $\calA$-linear left adjoints the claim follows from \cref{obs:finite-limits-of-linear-left-adjoints}.
\end{proof}

\begin{proof}[Proof of \cref{thm:sqz-general-case}]
    Hitting the the square from \cref{thm:sqz-stable-nonsplit-base} with $(-) \otimes_\calA \mcal{M}$ 
    produces by \cref{prop:almost-limit-tensoring} 
    an almost cartesian square of the desired form. 
    It remains to identify the essential image, or equivalently, the kernel of $\Phi^R$.
	The forgetful functor $\LMod_{A^\eta}(\mcal{M}) \to \mcal{M}$ is conservative so we may replace $\Phi^R$ with the composite 
	\[\Null_{\theta_\eta}(\LMod_A(\mcal{M})) \xrightarrow{\Phi^R} \LMod_{A^\eta}(\mcal{M}) \xrightarrow{\forget} \mcal{M}\]
	Tracing carefully through the construction of $\Phi$ we see that the above composite is precisely the right adjoint of the functor
	$\Phi(A\otimes(-),\gamma_A)$ 
	constructed in \cref{ex:weird-adjoint-sqz},
	which we computed in \cref{cor:weird-adjoint-special-case} to be:
	\[\Phi(A \otimes (-),\gamma_A)^R \colon (X,h) \longmapsto \fib\left(\beta_h \colon \Und(X) \to \Und(\Sigma I \otimes_A X)\right)\]
	We conclude that 
	$(X,h) \in \ker(\Phi^R)$ 
	if and only if $X$ is $\beta$-divisible.
	The characterization of the image now follows from \cref{lem:recollement-from-adjunction}.
\end{proof}

\subsection{Obstruction theory}

In this short section we add $t$-structures into the mix and prove a simpler variant of \cref{thm:sqz-general-case} which avoids any mention of the technical condition characterizing the image. 
We then deduce from it an obstruction theory for lifting modules along a square zero extension.

Let us henceforth fix a prestable presentably monoidal \category{} $\calA^{\ge 0}$ such that $\Sp(\calA^{\ge 0}) \simeq \calA$.
Let us write
$\LMod_{\calA^{\ge 0}}(\PrLpst) \subseteq \LMod_{\calA^{\ge 0}}(\PrL)$
for the full subcategory spanned by prestable presentable left $\calA^{\ge 0}$-modules.
Since stabilization
$\Sp(-) \colon \PrL \to \PrLst$ is monoidal it induces a functor:
\[\Sp(-)\colon \LMod_{\calA^{\ge 0}}(\PrLpst) \too \LMod_\calA(\PrLst), \qquad \calM^{\ge 0} \longmapsto \calM\]

\begin{obs}\label{obs:conn-implies-torsion}
    Let $(A,I,\eta)$ be a square zero datum in $\calA^{\ge 0}$, and let $\mcal{M}^{\ge 0} \in \LMod_{\calA^{\ge 0}}(\PrLpst)$. 
    Let $(X,h) \in \Null_{\theta_\eta}(\LMod_A(\mcal{M}))$ be $\beta$-divisible.
    We claim that if $X$ is bounded below, i.e. $n$-connected for some $n \in \bbZ$, then $X$ is in fact $\infty$-connected.
    Indeed, by assumption $\beta_h$ provides an equivalence $\Und(X) \simeq \Und(\Sigma I \otimes_A X)$ and thus
    \[ \conn(X) = \conn(\Sigma I \otimes_A X) \ge \conn(X)+\conn(\Sigma I)= \conn(X)+\conn(I)+1 \ge \conn(X)+1.\]
    In particular, if $\calM^{\ge 0}$ is separated and $X \in \calM^{\ge 0}$ then $X \simeq 0$.
\end{obs}

\begin{lem}\label{lem:Hurewitz-sqz}
    Let $(A,I,\eta)$ be a square zero datum in $\calA^{\ge 0}$, and let $\mcal{M}^{\ge 0} \in \LMod_{\calA^{\ge 0}}(\PrLpst)$.
    Then the functor $A \otimes_{A^\eta} (-) \colon \LMod_{A^\eta}(\mcal{M}) \to \LMod_{A}(\mcal{M})$ 
    detects connectivity, i.e.
    if $X \in \LMod_{A^\eta}(\mcal{M})$ is such that $A\otimes_{A^\eta}X$ is connective then $X$ is connective.
\end{lem}
\vspace{-1em}
\begin{proof}
    We begin by considering an "Adams type" filtration
    $X^\star \colon \bbZ^{\downarrow}_{\ge 0} \to \mcal{M}$,
    on $X$ by setting $X^k = X$ for $k = 0$ and
    $X^k\coloneqq I^{\otimes_{A^\eta} k} \otimes_{A^\eta} X$ for $k > 0$,
    with the maps $X^{k+1} \to X^k$ given by multiplication.
    By \cite[Proposition 7.4.1.14]{HA} the multiplication map 
    $I^{\otimes 2} \to I$ is null, hence the maps
    $X^{k+1} \to X^k$ are null for all $k \ge 1$.
    In particular we have 
    $\varprojlim_k X^k_{<0} \simeq 0$.
    Since truncation 
    $(-)_{<0} \colon \mcal{M} \to \mcal{M}_{<0}$
    preserves colimits, to prove $X$ is connective it suffices to show that $\cofib(X^k \to X)$ is connective for all $k \ge 0$.
    The associated graded is readily computed
    \begin{align*}
        \gr_k X^\star 
        & = \cofib(I^{\otimes_{A^\eta} k+1} \otimes_{A^\eta}X \to I^{\otimes_{A^\eta} k} \otimes_{A^\eta}X ) \\
        & \simeq A \otimes_{A^\eta} I^{\otimes_{A^\eta} k} \otimes_{A^\eta} X\\
        & \simeq A \otimes_{A^\eta} I^{\otimes_{A^\eta}k}\otimes_A (A \otimes_{A^\eta} X)
    \end{align*}
    where in the last step we used \cref{obs:square-bimodules-Aeta}.
    We see that $\gr_k X^\star$ is connective for all $k \ge 0$, and thus by induction on $k$ the same holds for $\cofib(X^k \to X^0=X)$.
\end{proof}

\begin{thm}\label{thm:connective-main-theorem}
    Let $(A,I,\eta)$ be a square zero datum in $\calA^{\ge 0}$, and let $\mcal{M}^{\ge 0} \in \LMod_{\calA^{\ge 0}}(\PrLpst)$.
    The fully faithful embedding of \cref{thm:sqz-general-case}  restricts to an equivalence:
    \[  \LMod_{A^\eta}(\mcal{M}^{\ge 0}) \simeq \Null_{\theta_\eta}\left(\LMod_A(\mcal{M}^{\ge 0})\right)\]
\end{thm}
\begin{proof}
    \cref{thm:sqz-general-case} applied to the $\calA$-module $\calM$ produces the following diagram
    \[\begin{tikzcd}
	{\LMod_{A^\eta}(\mcal{M}^{\ge 0})} & {\Null_{\theta_\eta}\left(\LMod_{A}(\calM^{\ge 0})\right)} & {\mcal{M}^{\ge 0}} \\
	{\LMod_{A^\eta}(\calM)} & {\Null_{\theta_\eta}\left(\LMod_{A}(\calM)\right)} & {\mcal{M}}
	\arrow[from=1-3, to=2-3]
	\arrow[from=1-2, to=1-3]
	\arrow[hook, from=1-2, to=2-2]
	\arrow[from=2-2, to=2-3]
	\arrow["{\Phi}", hook, from=2-1, to=2-2]
	\arrow[hook, from=1-1, to=1-2]
	\arrow[hook, from=1-1, to=2-1]
	\arrow["{A\otimes_{A^\eta} (-)}"', curve={height=18pt}, from=2-1, to=2-3]
    \end{tikzcd}\]
    where the left square consists entirely of fully faithful functors.
    The right square is evidently a pullback and the outer square is a pullback by  \cref{lem:Hurewitz-sqz}, hence the left square is a pullback.
    It remains to show that the top left functor is essentially surjective.
    Observe that for any $(X,h) \in \Null_{\theta_\eta}(\LMod_A(\calM))$ we have \[\conn(\Phi^R(X,h)) \ge \min \left\{\conn(I) + \conn(X),\conn(X)\right\}  \ge \conn(X)\]
    and thus the adjunction $\Phi \dashv \Phi^R$ restricts to connective objects:
    \[\Phi|_{\LMod_{A^\eta}(\calM^{\ge 0})} \colon  \LMod_{A^\eta}(\calM^{\ge 0})\adj \Null_{\theta_\eta}(\LMod_A(\calM^{\ge 0})) \colon \Phi^R|_{\Null_{\theta_\eta}(\LMod_A(\calM^{\ge 0}))} \]    To conclude we must show that the right adjoint
    is conservative.
    Note that $\Phi^R|_{\Null_{\theta_\eta}(\LMod_A(\calM^{\ge 0}))}$ is a colimit preserving functor between prestable \categories{} so it suffices to check it detects the zero object. 
    Namely, we must show that whenever $(X,h) \in \Null_{\theta_\eta}(\LMod_A(\calM^{\ge 0}))$ satisfies $\Phi^R(X,h) \simeq 0$ then $X \simeq 0$.
    Since $\Phi^R(X,h) \simeq 0$ if and only if $(X,h)$ is $\beta$-divisible this follows directly from \cref{obs:conn-implies-torsion}.
\end{proof}

\printbibliography[heading=bibintoc]

\end{document}